\documentclass[final]{siamltex}
\usepackage{graphicx}
\usepackage{color}
\usepackage{eufrak}      
\usepackage{dsfont}        

\newcommand{\EQ}{\begin{eqnarray}}
\newcommand{\EN}{\end{eqnarray}}
\newcommand{\EQQ}{\begin{eqnarray*}}
\newcommand{\ENN}{\end{eqnarray*}}
\newcommand{\nnum}{\nonumber}

\title{On the Optimal Control of  Impulsive Hybrid Systems  On  Riemannian Manifolds\thanks{This work was supported by  NSERC and AFOSR.}}

\author{Farzin Taringoo$^\dagger$\hspace{-.5cm}\and Peter E. Caines\thanks{Department of Electrical and Computer Engineering and Centre for Intelligent
Machines, McGill University, Montreal, Canada, \{taringoo, peterc@cim.mcgill.ca\}.}}

\begin{document}

\maketitle

\begin{abstract}
This  paper provides a geometrical derivation of the Hybrid Minimum Principle (HMP) for autonomous impulsive hybrid systems on Riemannian manifolds,  i.e. systems where the manifold valued
component of the hybrid state trajectory may have a jump discontinuity 
when
the discrete component changes value. The analysis  is expressed in terms of extremal trajectories on the cotangent bundle of the manifold state space. In the case of autonomous hybrid systems, switching manifolds are defined as smooth embedded submanifolds of the  state manifold and the jump function is defined as a smooth map on the switching manifold. The HMP results are obtained in the case of time invariant switching manifolds and state jumps on Riemannian manifolds.

\end{abstract}



\pagestyle{myheadings}
\thispagestyle{plain}
\markboth{Farzin Taringoo, Peter Caines}{On the Optimal Control of  Impulsive Hybrid Systems  On  Riemannian Manifolds }

\section{Introduction}
The problem of hybrid systems optimal control (HSOC) in Euclidean spaces has been studied
 in many papers, see e.g. \cite{Bengea,  Branicky,  vint2, vint1, sheet, Dmitruk, Piccoli, Gram, Boris,  Reidinger, Shaikh,Sussmann,Tomlin, Xu}. In particular, \cite{ Azhmyakov, Piccoli, Shaikh, Sussmann} present an extension of the Minimum Principle to hybrid systems
and \cite{Shaikh} gives an iterative algorithm which is based upon the Hybrid Minimum Principle (HMP) necessary conditions for both autonomous and controlled switching systems. In general the previously cited papers consider HSOC problems with a priori given sequences of discrete transitions. In addition, \cite{Piccoli} includes the case of switching costs. 

We note that historically optimal control theory has mainly used
the term Maximum Principle since optimal controls were derived via the
maximization of a Hamiltonian function, see e.g. \cite{Pontryagin}. However,
since we work with problems in the Bolza form we formulate the theory in terms of the minimization of a suitably defined Hamiltonian
function and consequently shall consistently use  the term Minimum 
Principle.

A geometric version of Pontryagin's Minimum Principle for a general class of state  manifolds is given in \cite{ Agra, Barbero, Sussmann}.
In this paper, we employ the control needle variation method  of  \cite{Barbero}, \cite{Bari} and \cite{ Taringoo7} to analyze  state variation propagation through switching manifolds  and hence we obtain a Hybrid Minimum Principle for autonomous hybrid systems (i.e. systems without controlled distinct state switchings) on Riemannian manifolds. It is shown that under appropriate hypotheses on the differentiability of the hybrid value function, the discontinuity of the adjoint variable at the optimal switching state and switching time is proportional to a differential form of the hybrid value function defined on the cotangent bundle of the state manifold. In the case of open control sets and Euclidean state spaces this result for impulsive hybrid systems  appeared in \cite{Reidinger2} without using the language of differential geometry.  We note that the analysis in this paper extends to the case of multiple autonomous switchings which has been treated in \cite{Shaikh} for hybrid systems defined on Euclidean spaces. 

The  continuity of the Hamiltonian function in the case of time invariant switching manifolds is derived in \cite{Shaikh} for open control value sets by employing the methods of the calculus of variations.  In this paper, for compact control value sets, we obtain  the continuity result  for the Hamiltonian function at the optimal switching time  by use of the  needle variation method. In particular we note that here the needle variation method  is generalized to a class of autonomous hybrid systems associated with  time varying embedded switching manifolds when the Hamiltonian function is discontinuous at optimal switching times. It is shown that the discontinuity is related to a differential form of an augmented hybrid value function. 

In this paper, in Section \ref{sec1} we give a general definition of hybrid systems on differentiable manifolds and then in Section \ref{sec2} present a geometric version of the Pontryagin Minimum Principle for optimal control systems. In Section \ref{sec3} we obtain the Hybrid Minimum Principle for impulsive hybrid systems using the method of needle variations. Complete proofs of the results of Section \ref{sec2} are given in the Appendices \ref{A0}-\ref{s1}. Furthermore the analysis of those cases where the hybrid value functions are differentiable, and  the switching manifolds and impulsive jumps are time varying, are given in the referenced  link \cite{link}.
\section{Hybrid Systems}
\label{sec1}
  In the following definition the standard hybrid systems framework (see e.g. \cite{Branicky,Shaikh}) is generalized to  the case where the continuous state space is a smooth manifold, where henceforth in this paper smooth means $C^{\infty}$. 
\begin{definition}
\label{d1}
A hybrid system with autonomous discrete transitions is a five-tuple 
\EQ \bf{H}:= \{\textsl{H}=\textsl{Q}\times \mathcal{M}, \textsl{U},\textsl{F},\mathcal{S},\mathcal{J}\}\EN
where:\\
$ Q=\{1,2,3,...,|Q|\}$
is a finite set of \textit{discrete (valued) states (components)} and $\mathcal{M}$ is a smooth $n$ dimensional  Riemannian continuous (valued) state (component) manifold with associated metric $g_{\mathcal{M}}$.
\\$H$  is called the \textit{hybrid state space} of $\bf{H}$.\\
$U\subset \mathds{R}^{u}$ is a set of \textit{admissible input control values}, where  $U$ is a compact set in $\mathds{R}^u$. The set of \textit{admissible input control functions} is $\mathcal{I}:=(L_{\infty}[t_{0},t_{f}),U)$, the set of all bounded measurable functions on some interval $[t_{0},t_{f}), t_{f}< \infty$, taking values in $U$.\\
$F$ is an \textit{indexed collection of smooth, i.e. $C^{\infty}$, vector fields} $\{f_{q_{i}}\}_{q_{i}\in Q}$, where $f_{q_{i}}:\mathcal{M}\times U\rightarrow T\mathcal{M}$ is a controlled vector field assigned to each discrete state; hence each $f_{q_{i}}$ is continuous on $\mathcal{M}\times U$ and continuously differentiable on $\mathcal{M}$ for all $u\in U$.\\
$\mathcal{S}:=\{n^{k}_{\gamma}: \gamma\in Q\times Q, 1\le k \le K<\infty,  n^{k}_{\gamma}\subset \mathcal{M}\}$ is a collection of  embedded time independent pairwise disjoint switching manifolds $($except in the case 
 where $\gamma = (p,q)$ is identified with $\gamma^{'} = (q,p)$$)$ such that for any ordered pair $\gamma=(p,q), n^{k}_{\gamma}$ is an open smooth, oriented codimension 1 submanifold of $\mathcal{M}$, possibly with boundary $\partial {n}^{k}_{\gamma}$. By abuse of notation, we describe the manifolds  locally by $n^{k}_{\gamma}=\{x: n^{k}_{\gamma}(x)=0, x\in \mathds{R}^{n}\}$.\\ 
$\mathcal{J}$  shall denote the family of the \textit{state jump functions} on the manifold $\mathcal{M}$. For an autonomous switching event from $p\in Q$ to $q\in Q$, the corresponding jump function is given by a smooth map $\zeta_{p,q}:\mathcal{M}\rightarrow \mathcal{M}$: if $x(t^{-})\in \mathcal{S}$ the state trajectory jumps to $x(t)=\zeta_{p,q}(x(t^{-}))\in \mathcal{M}$, $\zeta_{p,q}\in \mathcal{J}$. The non-jump special case is given by $x(t)=x(t^{-})$. \\
 We use the term \textit{impulsive hybrid systems} for those hybrid systems where the continuous part of the  state trajectory may have discontinuous transitions (i.e. jump) at controlled or autonomous discrete state switching times.
\end{definition}\\
We assume:\\
\textbf{\textit{A1}}: The \textit{initial state} $h_{0}:=(x(t_{0}),q_{0})\in H$ is such that $x_{0}=x(t_{0})\notin \mathcal{S}$ for all $q_{i}\in Q$. A \textit{(hybrid) input function }$u$ is defined on a half open interval
$[t_0, t_{f}), t_{f} \leq \infty$, where further $u\in \mathcal{I}$.
 A \textit{(hybrid) state trajectory} with initial state $h_{0}$ and (hybrid) input function $u$ is a triple $(\tau, q, x)$
consisting of a finite strictly increasing sequence of times (boundary and switching times)
$\tau = (t_0, t_1, t_2,\dots)$, an associated sequence of discrete states $q = (q_0,q_1, q_2, \dots)$, and a sequence $x(\cdot) = (x_{q_0}(\cdot), x_{q_1}(\cdot), x_{q_2}(\cdot), \dots)$ of absolutely continuous
functions $x_{q_i}: [t_i, t_{i+1}) \rightarrow \mathcal{M}$ satisfying the  continuous and discrete dynamics given by the following definition.
\begin{definition}
\label{d0}
The continuous dynamics of a hybrid system $\bf{H}$ with initial condition $h_{0}=(x_{0}, q_{0})$, input control function $u\in \mathcal{I}$ and hybrid state trajectory $(\tau, q, x)$ are specified piecewise in time via the mappings
\EQ\label{state}(x_{q_{i}},u):[t_i, t_{i+1})\rightarrow \mathcal{M}\times U,\hspace{.5cm} i=0,...,L,\quad 0<L<\infty, \EN
where $x_{q_{i}}(.)$ is an integral curve of  $ \label{state3}f_{q_{i}}(.,u(.)): \mathcal{M}\times[t_i, t_{i+1})\rightarrow T\mathcal{M}$ satisfying 
\EQ\label{state2} \dot{x}_{q_{i}}(t) = f_{q_{i}}(x_{q_{i}}(t), u(t)),
\quad a.e.  \: t \in [t_i, t_{i+1}),\nnum\EN
where $x_{q_{i+1}}(t_{i+1})$ is given recursively by 
\EQ \label{thm1}x_{q_{i+1}}(t_{i+1})=lim_{t\uparrow t^{-}_{i+1}}\zeta_{q_{i},q_{i+1}}(x_{q_{i}}(t)), \quad h_{0}=(q_{0},x_{0}),  t<t_{f}.\EN

The discrete autonomous switching dynamics are defined as follows:\\
For all $p,q$,  whenever an admissible hybrid system trajectory governed by the controlled vector field $f_{p}$ meets any given  switching manifold $n_{p,q}$ transversally, i.e. $f_{p}(x(t^{-}_{s}),t^{-}_{s}) \notin T_{x(t^{-}_{s})}\mathcal{S}$, there is an autonomous switching to the controlled vector field $f_{q}$, equivalently, discrete state transition $p\rightarrow q,\hspace{.2cm} p,q\in Q$.   Conversely, any autonomous discrete state transition corresponds to a transversal intersection. 

A system trajectory is not continued after a non-transversal intersection with
a switching manifold. Given the definitions and assumptions above, standard arguments  give the existence and uniqueness of a hybrid state trajectory $(\tau, q, x)$, with initial state $h_{0}\in H$ and input function $u\in \mathcal{I}$, up to $T,$ defined to be the least of an explosion time or an instant of non-transversal intersection with a switching manifold. \hspace{2cm} \\
\end{definition}

We adopt:\\
 
\textbf{\textit{A2}}:~~(Controllability) For any $ q\in Q$, all pairs of states $(x_{1}, x_{2})$ are mutually accessible in any given time period $[t_{0},t], t_{0}<t<t_{f}$, via the controlled vector field $\dot{x}_{q}(t) = f_{q}(x_{q}(t), u(t)), \hspace{.1cm}$ for some $u\in \mathcal{I}=(L_{\infty}[t_{0},t_{f}),U)$.\\ 

\textbf{\textit{A3}}:~~ $\{l_{q_{i}}\}_{q_{i}\in Q}$, is a family of \textit{loss functions} such that $l_{q_{i}}\in C^{k}(\mathcal{M}\times \textit{U};\mathds{R}^{+}),k\geq1$, and $h$ is a \textit{terminal cost function} such that  $h\in C^{k}(\mathcal{M};\mathds{R}^{+}),k\geq1$.\\ 

Henceforth, Hypotheses \textbf{\textit{A1}}-\textbf{\textit{A3}} will be in force unless otherwise stated. 
Let $L$ be the number of switchings and $ u\in \mathcal{I}$ then we define the \textit{hybrid cost function} as
\EQ \label{cost}\hspace{.4cm}J(t_{0},t_{f},h_{0};L,u):=\sum^{L}_{i=0}\int^{t_{i+1}}_{t_{i}}l_{q_{i}}(x_{q_{i}}(s),u(s))ds+h(x_{q_{L}}(t_{f})),\nnum\\ t_{L+1}=t_{f}<T, u\in \mathcal{I},\EN
where we observe the conditions above yield $J(t_{0},t_{f},h_{0};L,u)<\infty$.
\begin{definition}
\label{d2}
 For a hybrid system $\bf{H}$, given the data $(t_{0},t_{f},h_{0};L)$, the \textit{Bolza Hybrid Optimal Control Problem} (BHOCP) is defined as the infimization  of the hybrid cost function $J(t_{0},t_{f},h_{0};L, u)$ over  the hybrid input functions $u\in \mathcal{I}$, i.e. 
 \EQ\label{cost3} J^{o}(t_{0},t_{f},h_{0};L)=inf_{u\in\mathcal{I}}J(t_{0},t_{f},h_{0};L,u).\nnum\EN 
 \end{definition}
 \begin{definition}
 A \textit{Mayer Hybrid  Optimal Control Problem} (MHOCP) is defined as the special case of the BHOCP  where the cost function given in (\ref{cost}) is evaluated only on the terminal state of the system, i.e. $l_{q_{i}}=0,\hspace{.2cm}i=1,...,L$.
 \end{definition}\\

In general, different control inputs result in different sequences of discrete states of different cardinality. However,  in this paper,  we shall restrict the infimization to be over the  class of control functions, generically denoted $\mathcal{U}\subset\mathcal{I}$,  which generates an  a priori given sequence of discrete  transition events.\\\\
We adopt the following standard notation and terminology, see \cite{Lewis}.
The time dependent flow associated to a differentiable time independent vector field $f_{q_{i}}$ is a map $\Phi_{f^{u}_{q_{i}}}$ satisfying (where for economy of notation $f^{u}_{q_{i}}(.):=f_{q_{i}}(.,u(t))$):
\EQ  \Phi_{f^{u}_{q_{i}}}:[t_i, t_{i+1})\times [t_i, t_{i+1})\times \mathcal{M}\rightarrow \mathcal{M}, \quad (t,s,x)\rightarrow \Phi^{(t,s)}_{f^{u}_{q_{i}}}(x):= \Phi_{f^{u}_{q_{i}}}((t,s),x)\in \mathcal{M},\nnum\EN
where
 \EQ\label{1}\Phi^{(t,s)}_{f^{u}_{q_{i}}}:\mathcal{M}\rightarrow \mathcal{M},\quad \Phi^{(s,s)}_{f^{u}_{q_{i}}}(x)=x,\EN
\EQ  \frac{d}{dt}\Phi^{(t,s)}_{f^{u}_{q_{i}}}(x)|_{t}=f_{q_{i}}\big(\Phi^{(t,s)}_{f^{u}_{q_{i}}}(x(s))\big), \hspace{.2cm}t,s\in[t_{i},t_{i+1}).\EN

We associate $T \Phi_{ f^{u}_{q_{i}}}^{(t,s)}(.)$ to  $\Phi^{(t,s)}_{f^{u}_{q_{i}}}:\mathcal{M}\rightarrow \mathcal{M}$ via the push-forward of $\Phi^{(t,s)}_{f^{u}_{q_{i}}}$. 
\EQ   T\Phi_{ f^{u}_{q_{i}}}^{(t,s)}:T_{x}\mathcal{M}\rightarrow T_{\Phi_{ f^{u}_{q_{i}}}^{(t,s)}(x)}\mathcal{M}.\EN
Following \cite{Lewis},  the corresponding \textit{tangent lift} of $f^{u}_{q_{i}}(.)$ is the time dependent vector field $f^{T,u}_{q_{i}}(.)\in TT\mathcal{M}$ on $T\mathcal{M}$

\EQ \label{10}f^{T,u}_{q_{i}}(v_{x}):=\frac{d}{dt}|_{t=s} T \Phi_{ f^{u}_{q_{i}}}^{(t,s)}(v_{x}),\quad v_{x}\in T_{x}\mathcal{M}, \EN
which is given locally as 
\EQ \label{2} f^{T,u}_{q_{i}}(x,v_{x})= \left [f^{u,i}_{q_{i}}(x)\frac{\partial}{\partial x^{i}}+(\frac{\partial f^{u,i}_{q_{i}}}{\partial x^{j}}v^{j})\frac{\partial}{\partial v^{i}}\right ]^{n}_{i,j=1},\EN
and $T \Phi_{ f^{u}_{q_{i}}}^{(t,s)}(.)$ is evaluated on $v_{x}\in T_{x}\mathcal{M}$, see \cite{Lewis}. 
The following lemma gives the relation between the push-forward of $\Phi^{(t,s)}_{f_{q_{i}}}$ and the tangent lift introduced in (\ref{2}).
 For simplicity and uniformity of notation, we use $f_{q_{i}}$ instead of $f^{u}_{q_{i}}$.
 The following lemma is taken from \cite{Barbero} and its results are essential to obtain the Minimum Principle along the optimal trajectory for standard optimal control problems. In this paper we use the same results to obtain the HMP statement for hybrid systems.
\begin{lemma}[\cite{Barbero}]
\label{l1}
Consider $f_{q_{i}}(.,u(.)):\mathcal{M}\times I\rightarrow T\mathcal{M}, I=[t_{i},t_{i+1})$ as a time dependent vector field on $\mathcal{M}$ and $\Phi^{(t,s)}_{f_{q_{i}}}$ as its corresponding flow. The flow of $f^{T,u}_{q_{i}}$, denoted by $ \Psi:I\times I\times T\mathcal{M}\rightarrow T\mathcal{M}$, satisfies:
 \EQ \Psi(t,s,(x,v))=(\Phi^{(t,s)}_{f_{q_{i}}}(x),T\Phi^{(t,s)}_{f_{q_{i}}}(v))\in T\mathcal{M},\hspace{.5cm} (x,v)\in T\mathcal{M}.\nnum\EN
 \end{lemma}\hspace{8cm}
 
 \section{ The Pontryagin Minimum Principle for standard optimal control problems}
 \label{sec2}

In this section we focus on the \textit{Pontryagin Minimum Principle} (PMP) for standard (non-hybrid) optimal control problems defined on a Riemannian manifold $\mathcal{M}$. A standard optimal control problem (OCP) can be obtained from a BHOCP, see (\ref{cost}),  by fixing the discrete states  $q_{i}$ to $q$, and hence $L$ to the value 0. The resulting optimal control problem in Bolza form becomes that of the infimization of the cost (\ref{cost}) with respect to state dynamics which by suppressing notation of $q$ may be written $\dot{x}=f(x(t),u(t)),\hspace{.1cm} x(t)\in \mathcal{M}, u(t)\in \mathcal{U}, t\in[t_{0},t_{f}].$

 \subsection{The Relationship between Bolza and Mayer Problems}
 In Section 2 both the BHOCP and the MHOCP were introduced; since the results in this paper are only stated for the Mayer problem we now briefly explain the relationship between them.

In general (see \cite{Barbero}),  a Bolza problem can be converted to a Mayer problem with state variable $\hat{x}:=(x, x_{n+1})$ by adjoining an auxiliary state $x_{n+1}$ to the state $x$,  one then defines the dynamics to be given by
\EQ\label{Mayer} \dot{\hat{x}}(t)=\left[\begin{array}{cc}\dot{x}(t)\\\dot{x}_{n+1}(t) \end{array}\right]=\left[\begin{array}{cc}f(x(t),u(t))\\ l(x(t),u(t)) \end{array}\right],\EN
where $f$ and $l$ are  respectively the dynamics and the running cost of the Bolza problem.
Then the equivalent Mayer problem is obtained by the infimization of the penalty function $\hat{h}(.)$ defined  as follows: 
\EQ\label{Mayer2} \hat{h}(\hat{x}(t_{f}))\equiv \hat{h}(x(t_{f}),x_{n+1}(t_{f})):=x_{n+1}(t_{f})+h(x(t_{f}))= J(t_{0},t_{f},x_{0},u),\EN
where  $h$ is the terminal cost function of the Bolza problem.
Note that after such a transformation from a Bolza problem the state space of the resulting Mayer problem is $\mathcal{M}_{B}\times \mathds{R}$, where $\mathcal{M}_{B}$ is the state manifold  of the Bolza problem.


\subsection{Elementary Control and Tangent Perturbations}
\label{ss1}
We now present some results from \cite{Agra}, \cite{Barbero} and \cite{Lee}. It is essential to note that henceforth in this paper we treat the general Mayer problem with state space manifold denoted by $\mathcal{M}$. In the special case where Mayer OCP is derived from a Bolza problem $\mathcal{M}_{B}$ takes the product form given in the previous section.

Consider the nominal control  input $u(.)$ and define the associated perturbed control as 
\EQ \label{needle}u_{\pi(t^{1},u_{1})}(t,\epsilon)=\left\{ \begin{array}{cc} \quad u_{1} \quad t^{1}-\epsilon \leq t\leq t^{1},\\ u(t)\quad\mbox{elsewhere,} \end{array}\right. \EN
where $0\leq \epsilon<\infty, u_{1}\in U$. For brevity in notation $u_{\pi(t^{1},u_{1})}(t,\epsilon)$ shall be written $u_{\pi}(t,\epsilon)$.

Associated to $u_{\pi}(.,.)$ we have the corresponding state trajectory $x_{\pi}(.,.)$ on $\mathcal{M}$. It may be shown under suitable hypotheses, $lim_{\epsilon\rightarrow 0}x_{\pi}(t,\epsilon)=x(t) $ uniformly for $t_{0}\leq t \leq t_{f}$, see \cite{Piccoli} and  \cite{Lee}. 
Following (\ref{1}),  the flow resulting from the perturbed control is defined as:
\EQ \Phi_{\pi, f}^{(t, s), x}(.):[0,\tau]\rightarrow \mathcal{M},\quad x\in \mathcal{M},t,s\in [t_{0},t_{f}], \tau\in \mathds{R}^{+}, \Phi_{\pi, f}^{(t, s), x}(\epsilon)\in \mathcal{M},\nnum\EN
where $\Phi_{\pi, f}^{(t, s), x}(.)$ is the flow corresponding to the perturbed control $u_{\pi}(t,\epsilon)$, i.e. $\Phi_{\pi, f}^{(t, s), x}(\epsilon):=\Phi_{f^{u_{\pi}(t,\epsilon)}}^{(t, s)}(x(s))$.
The following lemma gives the formula of the variation of  $\Phi_{\pi, f}^{(t,s),x}(.)$ at the limit from the right $0^{+}:=lim_{\epsilon\downarrow 0} \epsilon$. 
We recall that the point $t^{1}\in (t_{0},t_{f})$ is called a Lebesgue point of  $u(.)$ if, (\cite{Agra}):
\EQ\lim_{s_{1}\downarrow t^{1}}\frac{1}{|s_{1}-t^{1}|}\int^{s_{1}}_{t^{1}}|u(\tau)-u(t^{1})|d\tau=0.\nnum \EN 
For any $u\in L_{\infty}([t_{0},t_{f}],U)$, $u$ may be modified on a set of measure zero so that all points are   Lebesgue points (see \cite{Rudin}, page 158, and \cite{Segal}) in which case, necessarily, the value of any cost function is unchanged.
\begin{lemma}[\cite{Barbero}]
\label{l2}
For a Lebesgue time $t^{1}$, the curve $ \Phi_{\pi, f}^{(t^{1},s),x}(.):=\Phi_{f^{u_{\pi}(t,\epsilon)}}^{(t^{1}, s)}(x(s)):[0,\tau]\rightarrow \mathcal{M}$ is differentiable from the right at $\epsilon=0$ and the corresponding tangent vector $\frac{d}{d\epsilon}\Phi_{\pi, f}^{(t^{1},s),x}|_{\epsilon=0}$ is given by
\EQ \label{12}\frac{d}{d\epsilon}\Phi_{\pi, f}^{(t^{1},s),x}|_{\epsilon=0}=f(x(t^{1}),u_{1})-f(x(t^{1}),u(t^{1}))\hspace{.1cm}\in\hspace{.1cm} T_{x(t^{1})}\mathcal{M}.\EN
\end{lemma}\hspace{8cm} \\
The tangent vector $ f(x(t^{1}),u_{1})-f(x(t^{1}),u(t^{1}))$ is called the \textit{elementary perturbation vector} associated to the perturbed control $u_{\pi}$ at $(x(t),t)$.
The displacement of the tangent vectors at $x\in \mathcal{M}$ is given by the push-forward of the vector field $f$, see sections below. 
\subsection{Adjoint Processes and the Hamiltonian}
In this section we present the  definitions of the adjoint process and the Hamiltonian function which appear in the statement of the Minimum Principle.  In the case $\mathcal{M}=\mathds{R}^{n}$,  by the smoothness  of $f$ we may define the following system of  differential equations:
\EQ \label{4} \dot{\lambda}^{T}(t)=-\lambda^{T}(t)\frac{\partial f}{\partial x}(x(t),u(t)),\hspace{.1cm} t\in[t_{0},t_{f}],\hspace{.1cm} x(t_{0})\in \mathds{R}^{n}.\EN  
 The matrix solution $\varphi$ of  $\dot{\varphi}(t)=\frac{\partial f}{\partial x}(x(t),u(t))\varphi(t),$ where $\varphi(0)=I,$ gives the transformation between tangent vectors on the state trajectory $x(t)$ from time $t^{1}$ to $t^{2}$ (see \cite{Lee}), in other words, considering $v_{1}$ as a tangent vector at $x(t^{1}),$ the push-forward of $v_{1}$  under  $\Phi^{(t^{2},t^{1})}_{f}$ is 
\EQ v_{2}=T\Phi^{(t^{2},t^{1})}_{f}(v_{1})=\varphi(t^{2}-t^{1})v_{1},\hspace{.2cm} v_{1}\in T_{x(t^{1})}\mathds{R}^{n}\simeq \mathds{R}^{n},t^{1},t^{2}\in[t_{0},t_{f}].\nnum\EN
Evidently the vector $v(t)=\phi(t)v(0)$ is the solution of the following differential equation:
\EQ\label{8} \dot{v}(t)=\frac{\partial f}{\partial x}(x(t),u(t))v(t),\quad v(0)=v_{0}, v(t)\in T_{x(t)}\mathds{R}^{n}\simeq \mathds{R}^{n}.\EN
A key feature of the solution of  (\ref{4}) is that along  $x(.)$, $\lambda^{T}(.)v(.)$ remains constant since
\EQ \label{111}\frac{d}{dt}(\lambda^{T}(t)v(t))=\dot{\lambda}^{T}(t)v(t)+\lambda^{T}(t)\dot{v}(t)=-\lambda^{T}(t)\frac{\partial f}{\partial x}v(t)+\lambda^{T}(t)\frac{\partial f}{\partial x}v(t)=0.\nnum\\\EN
For a general Riemannian manifold $\mathcal{M}$, the role of the adjoint process $\lambda$ is played by a trajectory in the cotangent bundle of $\mathcal{M}$, i.e.  $\lambda(t)\in T^{*}\mathcal{M}$. As in the definition of the tangent lift, we define the \textit{cotangent lift} which corresponds to the variation of a differential form $\alpha\in T^{*}\mathcal{M}$ (see \cite{Tyner}):
\EQ \label{6} f^{T^{*},u}(\alpha_{x}):=\frac{d}{dt}|_{t=-s}T^{*}\Phi^{(-t,s)}_{f^{u}}(\alpha_{x}),\quad \alpha_{x}\in T^{*}_{x}\mathcal{M},\EN
where $x=x(t)=\Phi^{(t,s)}_{f^{u}}(x(s))$.
As in (\ref{2}), in the local coordinates, $(x,p)$, of $T^{*}\mathcal{M}$, we have
\EQ \label{e} f^{T^{*},u}(x,p)=\left [f^{u,i}(x)\frac{\partial}{\partial x^{i}}-(\frac{\partial f^{u,i}}{\partial x^{j}}p^{j})\frac{\partial}{\partial p^{i}} \right ]^{n}_{i,j=1},\EN
where $T^{*}\Phi^{(-t,s)}_{f^{u}}(.)$ is the pull back of $\Phi^{(-t,s)}_{f^{u}}$ applied to differential forms $\alpha_{x}\in T^{*}_{x}\mathcal{M}$. The minus sign  in front of $t$ in (\ref{6}) is due to the fact that pull backs act in the opposite sense to push forwards, therefore the variation of a  covector $\alpha_{x}$  at $x=x(s)$ depends upon $\Phi^{-1}$ which notationally corresponds to $-t$, see \cite{Tyner}.   
The following lemma gives the connection between the cotangent lift defined in (\ref{6}) and its corresponding flow on $T^{*}\mathcal{M}$. Let ($T^{*}\Phi^{(t,s)}_{f})^{-1}=T^{*}\Phi^{(-t,s)}_{f}$, the pull back of  $\Phi^{-1}$, whose existence is guaranteed since $\Phi:\mathcal{M}\rightarrow \mathcal{M}$ is a diffeomorphism, see \cite{Tyner}.

\begin{lemma}[\cite{Barbero}]
\label{l3}
Consider $f(x(t),u(t))$ as a time dependent vector field on $\mathcal{M}$, then the flow  
$ \Gamma: I\times I\times T^{*}\mathcal{M}\rightarrow T^{*}\mathcal{M},$ $I=[t_{0},t_{f}]$, satisfies 
\EQ\label{7} \Gamma(t,s,(x,p))=(\Phi^{(t,s)}_{f}(x),(T^{*}\Phi^{(t,s)}_{f})^{-1}(p)), \hspace{.1cm}(x,p)\in T^{*}\mathcal{M},\EN
and $\Gamma$ is the corresponding integral flow of  $f^{T^{*},u}$.

\end{lemma}\hspace{8cm}\\
 We now generalize  (\ref{4}) and (\ref{8})   to  differentiable manifolds.
Along a given trajectory $\lambda(.)\in T^{*}\mathcal{M}$, the variation with respect to time, $\dot{\lambda}(t)$, is an element of $TT^{*}\mathcal{M}$. The vector field defined  in (\ref{6}) is thus the mapping $ f^{T^{*},u}:T^{*}\mathcal{M}\rightarrow TT^{*}\mathcal{M}$, which generalizes  (\ref{4}) to a mapping from $\lambda(t)\in T^{*}\mathcal {M}$ to $\dot{\lambda}(t)\in TT^{*}\mathcal{M}$.
The generalization of (\ref{111}) to $\mathcal{M}$ is given in the following proposition.
\begin{proposition}[\cite{Barbero}]
\label{p1}
Let $f_{q}(.,u(.)): \mathcal{M}\times I\rightarrow T\mathcal{M},\quad I=[t_{0},t_{f}],$ be a time dependent vector field giving rise to the associated  pair $f^{T,u},f^{T^{*},u}$; then along an integral curve of $f(.,u)$ on $\mathcal{M}$ 
\EQ \langle\Gamma,\Psi\rangle:I\rightarrow \mathds{R},\nnum\EN
is a constant map, where $\Gamma$ is an integral curve of $f^{T^{*},u}$ in $T^{*}\mathcal{M}$ and $\Psi$ is an integral curve of $f^{T,u}$ in $T\mathcal{M}$.
\end{proposition}\hspace{8cm} \\

The integral curves $\Gamma$ and $\Psi$ are the generalizations of   $\lambda(.)$ and $v(.)$  in (\ref{8}) and (\ref{111}) in $\mathds{R}^{n}$ to the case of a differentiable manifold  $\mathcal{M}$.
The corresponding variation  of the elementary tangent perturbation  in  Lemma \ref{l2} is given in the following proposition.
\begin{proposition}[\cite{Barbero}]
\label{p2}
Let $\Psi:[t^{1}, t_{f}]\rightarrow T\mathcal{M}$ be the integral curve of $f^{T,u}$ with the initial condition $\Psi(t^{1})=[f(x(t^{1}),u_{1})-f(x(t^{1}),u(t^{1}))]\in T_{x(t^{1})}\mathcal{M}$, then 
\EQ \frac{d}{d\epsilon}\Phi_{\pi, f}^{(t,t^{1}),x}|_{\epsilon=0}=\Psi(t),\quad t\in [t^{1},t_{f}].\nnum\EN
\end{proposition}\hspace{8cm} \\
By the result above and Lemma \ref{l1} we have
\EQ\frac{d}{d\epsilon}\Phi_{\pi, f}^{(t,t^{1}),x}|_{\epsilon=0}=T\Phi^{(t,t^{1})}_{f}([f(x(t^{1}),u_{1})-f(x(t^{1}),u(t^{1}))])\hspace{.1cm}\in\hspace{.1cm} T_{x(t^{1})}\mathcal{M}.\nnum\EN
\subsection{Hamiltonian Functions and Vector Fields}
Here we recall the notions of Hamiltonian vector fields (see e.g.   \cite{Arnold}), which were employed in \cite{Agra} to obtain a Minimum Principle for optimal control problems in a geometrical framework.


For an optimal  (non-hybrid) control problem defined on the state manifold $\mathcal{M}$, with controlled vector field $f(x(t),u(t))\in T_{x(t)}\mathcal{M}$, the Hamiltonian function for the Mayer problem is defined as:
\EQ\label{ham1} H: T^{*}\mathcal{M}\times U\rightarrow \mathds{R},\EN
\EQ\label{ham2} H(p, x, u)=\langle p,f(x,u)\rangle,\quad p\in T^{*}_{x}\mathcal{M}, \hspace{.2cm}f(x,u)\hspace{.1cm}\in\hspace{.1cm} T_{x}\mathcal{M}.\EN
In general, the Hamiltonian is a smooth function $H\in C^{\infty}(T^{*}\mathcal{M})$  with an  associated  Hamiltonian vector field  $\overrightarrow{H}\in \mathfrak{X}(T^{*}\mathcal{M})$ defined by (see \cite{Agra})
\EQ \omega_{\lambda}(.,\overrightarrow{H})=dH,\quad \lambda\in T^{*}\mathcal{M},\nnum\EN
where $\omega_{\lambda}\in \Omega^{2}(T^{*}\mathcal{M})$ is the symplectic form (see e.g. \cite{Fland}, \cite{Lee2}) defined on $T^{*}\mathcal{M}$ (see \cite{Agra,Jurd}) and $\mathfrak{X}(T^{*}\mathcal{M})$ is the space of smooth vector fields defined on  $T^{*}\mathcal{M})$.
The Hamiltonian vector field satisfies $i_{\overrightarrow{H}}\omega_{\lambda}=-dH$, (see \cite{Agra}) where $i_{\overrightarrow{H}}$ is the contraction mapping (see \cite{ jost, Lee2}) along the vector field $\overrightarrow{H}$.
In the local coordinates $(x,p)$ of $T^{*}\mathcal{M}$, we have:
\EQ dH=\sum^{n}_{i=1}\frac{\partial H}{\partial p^{i}}dp^{i}+\frac{\partial H}{\partial x^{i}}dx^{i},\quad \overrightarrow{H}=\sum^{n}_{i=1}\frac{\partial H}{\partial p^{i}}\frac{\partial }{\partial x^{i}}-\frac{\partial H}{\partial x^{i}}\frac{\partial }{\partial p^{i}}.\EN
So the Hamiltonain system $\dot{\lambda}(t)=\overrightarrow{H}(\lambda),\hspace{.1cm} \lambda\in T^{*}\mathcal{M}$ is locally written as:
\EQ\label{hamham}&&\hspace{2cm}\dot{x}(t)=\frac{\partial H}{\partial p^{i}}, \hspace{2cm} \dot{p}(t)= -\frac{\partial H}{\partial x^{i}},\nnum\\&& \hspace{.2cm}\mbox{where}\hspace{.5cm} \lambda(t)=(x(t),p(t))\in T^{*}\mathcal{M},\hspace{.2cm} x(t_{0})=x_{0}, \lambda(t_{f})=dh(x(t_{f}))\in T^{*}_{x(t_{f})}\mathcal{M},\nnum\EN
where \EQ dh= \sum^{n}_{i=1}\frac{\partial h}{\partial x^{i}}dx^{i}\in  \Omega^{1}(\mathcal{M}).\EN

\subsection{Pontryagin Minimum Principle}
For standard (non-hybrid) optimal control problems defined on a Riemannian manifold $\mathcal{M}$ we have the following result known as Pontryganin Minimum Principle.
\begin{theorem}[\cite{Lee}]
Consider an OCP satisfying hypotheses  \textbf{\textit{A1-A3}} ($L=0, q_{i}=q$) defined on a Riemannian manifold $\mathcal{M}$. Then  corresponding to  an  optimal  control and optimal state trajectory pair, $(u^{o}, x^{o})$  there exists a nontrivial adjoint trajectory $\lambda^{o}(.)=(x^{o}(.),p^{o}(.))\in T^{*}\mathcal{M},$ defined along the optimal state trajectory, such that:
\EQ \hspace{.6cm}H(x^{o}(t),p^{o}(t),u^{o}(t))\leq H(x^{o}(t),p^{o}(t),v), \hspace{.2cm}\forall v\in U, t\in[t_{0},t_{f}],\nnum \EN
and the corresponding optimal adjoint trajectory $\lambda^{o}(.)\in T^{*}\mathcal{M}$ satisfies:
\EQ \dot{\lambda^{o}}(t)=\overrightarrow{H}(\lambda^{o}(t)),\quad t\in [t_{0},t_{f}].\nnum\EN

\end{theorem}

The Minimum Principle gives necessary conditions for optimality; conditions under which the Minimum Principle is sufficient for optimality are given in \cite{Bolt} and  \cite{vint3}.


\section{The Hybrid Minimum Principle for Autonomous Impulsive Hybrid Systems}
\label{sec3}
Here we consider a simple  impulsive autonomous hybrid system consisting of one switching manifold.  Consider a hybrid system with a single switching from the discrete state $q_{0}$ to the discrete state $q_{1}$ at the unique switching time $t_{s}$ on the optimal trajectory $(x^{o}(.),u^{o}(.))$  associated with the dynamics:
\EQ \dot{x}_{q_0}(t) = f_{q_0}(x(t), u(t)), \quad a.e.  \: t \in [t_{0}, t_{s}),\nnum\EN
\EQ \dot{x}_{q_1}(t) = f_{q_1}(x(t), u(t)), \quad a.e.  \: t \in [t_s, t_{f}],\nnum\EN
where $x(t_{0})=x_{0}, t_{s}=t_{1}, t_{f}=t_{2}$ and 
\EQ f_{q_{i}}(.,u(.)): \mathcal{M}\times [t_i, t_{i+1})\rightarrow T\mathcal{M},\quad i=0,1,\nnum\EN
together with a smooth state jump $\zeta:=\zeta_{q_{0},q_{1}}:\mathcal{M}\rightarrow \mathcal{M} $ with the following action:
\EQ x^{o}(t_{s})=\zeta(x^{o}(t^{-}_{s}))=lim_{t\rightarrow t^{-}_{s}}\zeta(x(t)),\quad x^{o}(t^{-}_{s})\in \mathcal{S}\subset \mathcal{M}.\nnum\EN
We shall assume the switching manifold $\mathcal{S}$ is an embedded $n-1$ dimensional submanifold $\mathcal{S}:=n_{q_{0},q_{1}}$ which consists of a single switching manifold (see Section 2). 
Following \cite{Shaikh}, the control needle variation analysis is performed in two distinct cases. In the first case, the variation is applied after the optimal switching time, therefore there is no state variation propagation along the state trajectory before the switching manifold,  while in the second case, the control needle variation is applied  before the optimal switching time. In this case there exists a state variation propagation along the state trajectory which passes through  the switching manifold, see \cite{Shaikh} (see Figure \ref{11}).
\begin{figure}
\begin{center}
\hspace*{-4cm}\includegraphics[scale=.5]{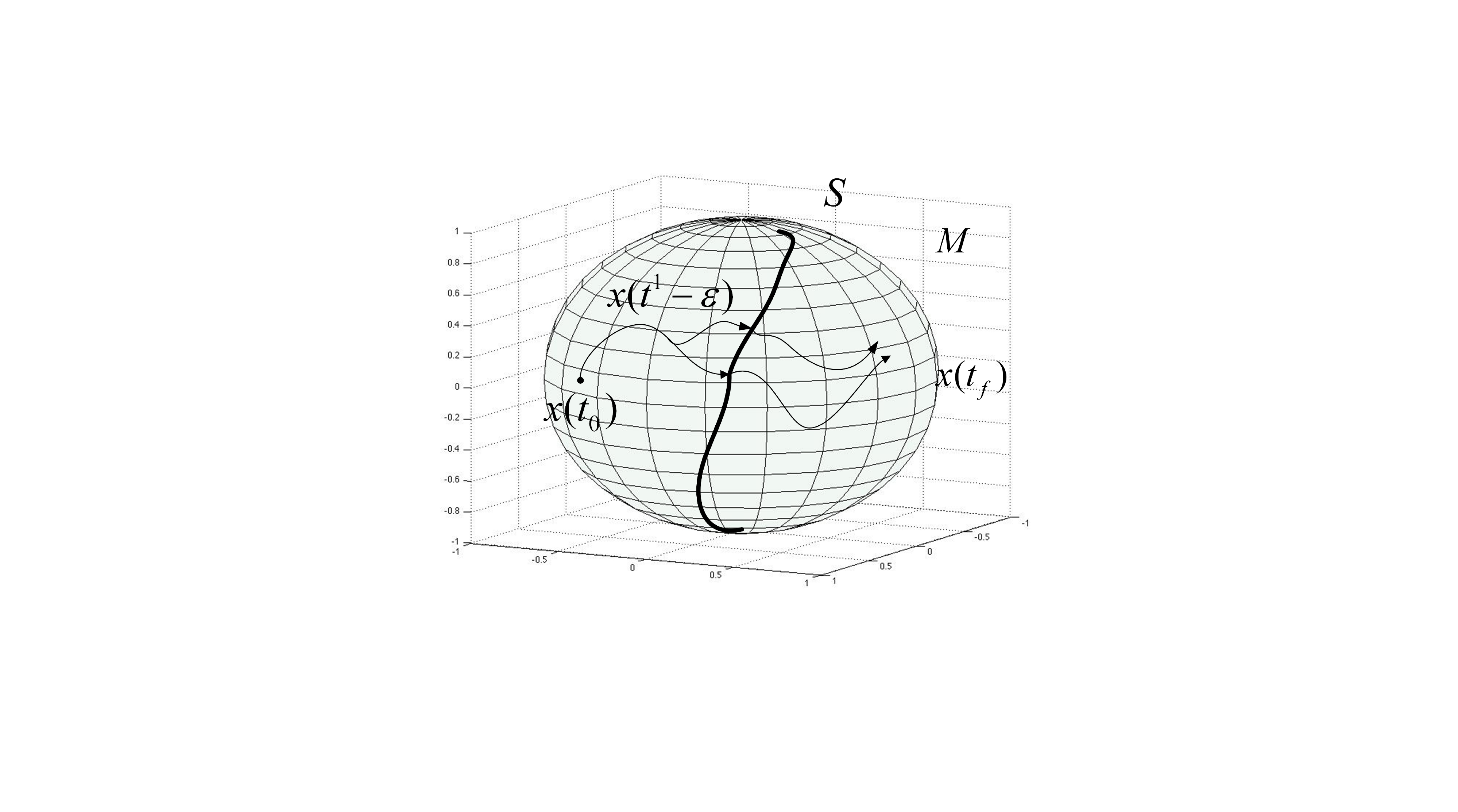}
\vspace*{-3cm}
      \caption{ Hybrid State Trajectory On the Sphere}
      \label{11}
      \end{center}
   \end{figure}  
\\
Recalling  assumption \textbf{\textit{A2}} in the Bolza problem and assuming the existence of optimal controls for each pair of given switching state and switching time, let us define  a function $v:\mathcal{M}\times (t_{0},t_{f})\rightarrow \mathds{R}$  for a hybrid system with one autonomous switching, i.e. $L=1$, as follows:
\EQ \label{v}v(x,t)=inf_{u\in \mathcal{U}}J(t_{0},t_{f},h_{0}, u),\EN
where 
\EQ x=\Phi^{(t^{-},t_{0})}_{f_{q_{0}}}(x_{0})\in \mathcal{S}\subset \mathcal{M}.\nnum\EN

\subsection{Non-Interior Optimal Switching States}
In this subsection, we show that the optimal switching state  for an MHOCP derived from a BHOCP (see (\ref{Mayer})) cannot be an interior point of  the \textit{attainable switching set} $\mathcal{A}(x_{0}, t_{s})\subset \mathcal{S},\hspace{.2cm}t_{0}<t_{s}<t_{f},$  for an MHOCP which is defined as
\EQ \mathcal{A}(x_{0}, t_{s})=\big\{x\in \mathcal{S}\hspace{.2cm}s.t.\hspace{.2cm}\exists u\in \mathcal{U}, \Phi^{(t^{-}_{s},t_{0})}_{f^{u}_{q_{0}}}(x_{0})=x\big\}.\nnum\EN 

Note that the state manifold of a Mayer problem derived from a Bolza problem is $\mathcal{M}_{B}\times \mathds{R}$ where $\mathcal{M}_{B}$ is the state manifold of the Bolza problem. In this paper, for simplicity and uniformity of notation,  the state manifold and the switching manifold of a Mayer problem shall also be denoted by $\mathcal{M}$ and $\mathcal{S}$ respectively. 
\begin{lemma}
\label{l4}
Consider an MHOCP derived from a BHOCP as in (\ref{Mayer}), (\ref{Mayer2})  with a single switching from the discrete state $q_{0}$ to the discrete state $q_{1}$ at the unique switching time $t_{s}$ on the optimal trajectory $(x^{o}(.), u^{o}(.))$ and an $n$ dimensional  switching manifold $\mathcal{S}=\mathcal{S}_{B}\times \mathds{R}:=n_{q_{0},q_{1}}$ defined in an $n+1$ dimensional manifold $\mathcal{M}=\mathcal{M}_{B}\times\mathds{R}$, where $\mathcal{S}_{B}\subset \mathcal{M}_{B}$ is the switching manifold of the BHOCP. Then an optimal switching state $x^{o}(t^{-}_{s})\in \mathcal{S}$ at the optimal switching time $t_{s}$ cannot be an interior point of $\mathcal{A}(x_{0}, t_{s})$ in the induced topology of  $\mathcal{S}$ from $\mathcal{M}$. 
\end{lemma}
\begin{proof} If $\mathcal{A}(x_{0}, t_{s})$ has empty interior in the topology induced on $\mathcal{S}$ from $\mathcal{M}$ the result is immediate. 
 Assume $x^{o}(t^{-}_{s})$ is an interior point of $\mathcal{A}(x_{0}, t_{s})$, i.e. there exists an open neighbourhood $B_{x^{o}(t^{-}_{s})}\subset \mathcal{A}(x_{0}, t_{s})$ of $x^{o}(t^{-}_{s})\in \mathcal{S}$.  Let us denote a coordinate system around $x^{o}(t^{-}_{s})$ by $(x^{o}_{1},...,x^{o}_{n+1}),$ where $x^{o}_{n+1}$ corresponds to the running cost of the Bolza problem, see (\ref{Mayer}).
Since the switching manifold of the MHOCP is defined by $\mathcal{S}=\mathcal{S}_{B}\times \mathds{R}$,  we may choose a neighbourhood $B_{x^{o}(t^{-}_{s})}$ of $x^{o}(t^{-}_{s})$ in the induced topology of $\mathcal{S}$ with the last coordinate $x_{n+1}$ free to vary in an open set in $\mathds{R}$. Hence fixing $x^{o}_{1}(t^{-}_{s}),...,x^{o}_{n}(t^{-}_{s})$, there exists $y\in B_{x^{o}(t^{-}_{s})}$ such that 
\EQ y_{i}=x^{o}_{i}(t^{-}_{s}),\hspace{.2cm} i=1,...,n,\quad y_{n+1}<x^{o}_{n+1}(t^{-}_{s}),\nnum\EN
which is accessible  by $f_{q_{0}}$ subject to a new control $\hat{u}(t),\hspace{.2cm}t_{0}\leq t< t_{s}$, where $\hat{u}$ is not necessarily equal to $u^{o}$.  Set the control $u(t)=u^{o}(t),\hspace{.2cm} t_{s}\leq t\leq t_{f}$; then $u(.)$ results in  an identical state trajectory on $[t_{s},t_{f}]$ for the Bolza problem (since the variables $x_{1},...,x_{n}$ do not change). However, the final hybrid cost corresponding to the new switching state $y$ is
\EQ J(t_{0},t_{f},(x_{0},q_{0});1,(u,u^{o}))=y_{n+1}+\int^{t_{f}}_{t_{s}}l_{1}(x^{o}(t),u^{o}(t))dt+h(x^{o}(t_{f}))<v(x^{o}(t^{-}_{s}),t_{s}),\nnum\EN
where $y_{n+1}=\int^{t_{s}}_{t_{0}}l_{0}(\hat{x}(t),\hat{u}(t))dt<x^{o}_{n+1}=\int^{t_{s}}_{t_{0}}l_{0}(x^{o}(t),u^{o}(t))dt$, 
 contradicting the optimality of $x^{o}(t^{-}_{s})$; we conclude  $x^{o}(t^{-}_{s})$ lies on the boundary of $\mathcal{A}(x_{0}, t_{s})$.
 \end{proof}

However the lemma above implies that the hybrid value function defined by (\ref{v}) cannot be differentiated in all directions  at the optimal switching state for MHOCPs derived from BHOCPs. Hence the main HMP Theorem \ref{t2} for MHOCPs below applies in potential to all MHOCPs derived from BHOCPs. The general HMP statement given below employs a differential form $dN_{x}$ corresponding to the normal vector to the switching manifold $\mathcal{S}\subset\mathcal{}M$ at the optimal switching state $x^{o}(t_{s})$. Now in the special case where the value function can be differentiated in all directions at $x^{o}(t_{s})\in \mathcal{S}$, it may be shown that $dN_{x^{o}(t_{s})}=\mu dv(x^{o}(t_{s}),t_{s})$ for some scalar $\mu$, see \cite{link}, Lemma A.1; this fact has significant implications for the theory of HMP as is shown in \cite{Taringoo1, Taringoo2, Taringoo3}.

\subsection{Preliminary Lemmas}
In order to use the methods  introduced in \cite{Agra, Barbero, Lee}, we establish Lemma \ref{44} using the perturbed control $u_{\pi}(.,.)$ and the associated state variation at the final state $x^{o}(t_{f})$.
 Denote by $t_{s}(\epsilon)$ the switching time corresponding to $u_{\pi}(t,\epsilon)$. Note that, in general, $\Phi^{(t,t_{0})}_{\pi, f_{q}}(x_{0})$ does  not necessarily intersect the switching manifold at $t_{s}$. Hence, we introduce the following perturbed control to guarantee that eventually the state trajectory meets the switching manifold.
 \EQ \label{u} u_{\pi}(t,\epsilon)=\left\{ \begin{array}{cc} \quad \hspace{-1cm}u^{o}(t) \quad \hspace{.2cm}t\leq t^{1}-\epsilon\\ \hspace{.2cm}u_{1}\quad  \hspace{.6cm}t^{1}-\epsilon\leq t\leq t^{1}\\ \hspace{-.5cm}u^{o}(t) \quad \hspace{.2cm}t^{1}< t \leq t_{s}\\u^{o}(t_{s})\quad t_{s}\leq t < t_{s}(\epsilon) \end{array}\right.,\EN
 The following lemma shows that under the control above, the hybrid state trajectory always intersects the switching manifold for sufficiently small $\epsilon\in \mathds{R}^{+}$.
 \begin{lemma}
 \label{44}
 For an MHOCP  satisfying \textbf{\textit{A1-A3}} with a single switching from the discrete state $q_{0}$ to the discrete state $q_{1}$ at the unique switching time $t_{s}$ on the optimal trajectory $(x^{o}(.),u^{o}(.))$,  the state trajectory associated to the control needle variation  $u_{\pi}(t,\epsilon)$ in (\ref{u}) intersects the $n-1$ dimensional switching manifold $\mathcal{S}\subset \mathcal{M}$ for all sufficiently small $\epsilon\in \mathds{R}^{+}$ and the corresponding  switching time $t_{s}(\epsilon)$ is differentiable with respect to $\epsilon$.
 
  \end{lemma}
  
  \begin{proof}
  The proof is given in Appendix \ref{A0}.
  \end{proof}
 \\
 


 \begin{lemma}
 \label{l555}
 For an MHOCP satisfying hypotheses  \textbf{\textit{A1-A3}} with a single switching from the discrete state $q_{0}$ to the discrete state $q_{1}$ at the unique switching time $t_{s}$ on the optimal trajectory $(x^{o}(.),u^{o}(.))$, the state variation at the switching time 
 $t_{s}$, i.e. $\frac{d \Phi_{\pi, f_{q_{1}}}^{(t_{s}(\epsilon),t^{1}),x} }{d\epsilon}|_{\epsilon=0}$, is given as follows:
 \EQ\frac{d \Phi_{\pi, f_{q_{1}}}^{(t_{s}(\epsilon),t^{1}),x} }{d\epsilon}|_{\epsilon=0}&\hspace{-.1cm}=&\hspace{-.1cm}T\zeta\circ T\Phi^{(t^{-}_{s},t^{1})}_{f_{q_{0}}}[f_{q_{0}}(x^{o}(t^{1}),u_{1})-f_{q_{0}}(x^{o}(t^{1}),u^{o}(t^{1}))]\nnum\\&&\hspace{-.4cm}+(\frac{d t_{s}(\epsilon)}{d \epsilon}|_{\epsilon=0}).\Big(T\zeta\big[f_{q_{0}}(x^{o}(t^{-}_{s}),u^{o}(t^{-}_{s}))\big]-f_{q_{1}}(x^{o}(t_{s}),u^{o}(t_{s}))\Big),\nnum\\&&\hspace{7.3cm}t^{1}\in [t_{0},t_{s}).\nnum\EN
\end{lemma}
\begin{proof}
The proof is obtained by the differentiation of the state flow combination; it is given in Appendix \ref{A}.
\end{proof}
 
 The following lemma gives a variational inequality as a necessary condition for the minimality of the Mayer hybrid cost function $h(x(t_{f}))=J(t_{0},t_{f},x_{0},u)$ defined by (\ref{Mayer2}). This inequality enables us to construct an adjoint curve $\lambda\in T^{*}\mathcal{M}$ which satisfies the HMP equations.
 
In order to prove the following lemma  we use the Taylor expansion of a smooth function defined on a Riemannian manifold, see \cite{Alvarez} and \cite{Smith}. 
For a given smooth function $h:\mathcal{M}\rightarrow R$ and a vector field $X\in \mathfrak{X}(\mathcal{M}),$ where $\mathfrak{X}(\mathcal{M})$ defines the space of all smooth vector fields on $\mathcal{M}$, the Taylor expansion of $h$ around $p\in \mathcal{M}$ along a tangent vector $X_{p}\in T_{p}\mathcal{M}$ is given by (see \cite{Smith}):
\EQ \label{tay}\hspace{-0cm}h(exp_{p}\theta X_{p})&\hspace{-.1cm}=&\hspace{-.1cm}h(p)+\theta (\nabla_{X}h)(p)+...+\frac{\theta^{n-1}}{(n-1)!}\times(\nabla^{n-1}_{X}h)(p)\nnum\\&&+\hspace{-0cm}\frac{\theta^{n}}{(n-1)!}\int^{1}_{0}(1-t)^{n-1}(\nabla^{n}_{X}h)(exp_{p}t\theta X)dt,\quad 0<\theta<\theta^{*},\nnum\\\EN
where $exp_{p}\theta X_{p}$ is the geodesic emanating from $p\in \mathcal{M}$ with the velocity $ X_{p}\in T_{p}\mathcal{M}$, $X(p)=X_{p}$ and $\theta^{*}$ is the upper bound of the existence of geodesics on the Riemannian manifold $\mathcal{M}$. The existence of  $\theta^{*}$ is guaranteed by the fundamental theorem of existence and uniqueness of geodesics (see \cite{jost}).
In (\ref{tay}), $\nabla:\mathfrak{X}(\mathcal{M})\times \mathfrak{X}(\mathcal{M})\rightarrow\mathfrak{X}(\mathcal{M})$ is the \textit{Levi-Civita} connection on $\mathcal{M}$ which satisfies the following characteristic relations:

\EQ Xg_{\mathcal{M}}(Y,Z)=g_{\mathcal{M}}(\nabla_{X}Y,Z)+g_{\mathcal{M}}(Y,\nabla_{X}Z),\quad \forall X,Y,Z\in \mathfrak{X}(\mathcal{M}),\nnum\EN
\EQ  \label{levi} \hspace{-.3cm}(i): \nabla_{X}Y-\nabla_{Y}X=[X,Y],\hspace{.5cm}(ii): \nabla_{X}f=X(f)\quad \forall X,Y\in \mathfrak{X}(\mathcal{M}), f\in C^{\infty}(\mathcal{M}).\nnum\\\EN
Based on the fundamental theorem of existence of geodesics on $\mathcal{M}$ (see \cite{jost}), for each $v_{\pi}(t_{f})\in T_{x(t_{f})}\mathcal{M}$ there exists a geodesic  emanating from $x(t_{f})$ with the velocity $v_{\pi}(t_{f})$. 

  \begin{lemma}
\label{l6}
For an  MHOCP  satisfying  \textbf{\textit{A1-A3}} with a single switching from the discrete state $q_{0}$ to the discrete state $q_{1}$ at the unique switching time $t_{s}$ on the optimal trajectory $(x^{o}(.),u^{o}(.))$,
\EQ \label{inq}\langle dh(x^{o}(t_{f})),v_{\pi}(t_{f})\rangle\geq 0,\quad \forall v_{\pi}(t_{f})\in K_{t_{f}},\EN
where
\EQ K_{t_{f}}=K^{1}_{t_{f}}\cup K^{2}_{t_{f}},\EN
and where 
\EQ\label{cone1} K^{1}_{t_{f}}=\bigcup_{t_{s}\leq t\leq t_{f}}\bigcup_{u_{1}\in U}T\Phi^{(t_{f},t)}_{f_{q_{1}}}[f_{q_{0}}(x^{o}(t),u_{1})-f_{q_{0}}(x^{o}(t),u^{o}(t))]\subset T_{x^{o}(t_{f})}\mathcal{M},\nnum\\\EN
and
\EQ\label{cone11}&&\hspace{-.1cm} K^{2}_{t_{f}}=\hspace{-.1cm}\bigcup_{t_{0}< t<t_{s}}\bigcup_{u_{1}\in U}T\Phi^{(t_{f},t_{s})}_{f_{q_{1}}}\circ T\zeta\circ T\Phi^{(t^{-}_{s},t)}_{f_{q_{0}}}[f_{q_{0}}(x^{o}(t),u_{1})-f_{q_{0}}(x^{o}(t),u^{o}(t))] \nnum\\&&\hspace{.6cm}+\big(\frac{dt_{s}(\epsilon)}{d\epsilon}|_{\epsilon=0}\big)T\Phi^{(t_{f},t_{s})}_{f_{q_{1}}}\Big(T\zeta \big[f_{q_{0}}(x^{o}(t^{-}_{s}),u^{o}(t^{-}_{s}))\big]-f_{q_{1}}(x^{o}(t_{s}),u^{o}(t_{s}))\Big) \nnum\\&&\hspace{.6cm}\subset T_{x^{o}(t_{f})}\mathcal{M}.\EN
 \end{lemma}

\begin{proof}
To apply (\ref{tay}) to  $h$, one needs to extend $v_{\pi}(t_{f})\in T_{x(t_{f})}\mathcal{M}$ to a smooth vector field denoted by $\tilde{\mathcal{V}}_{\pi}\in \mathfrak{X}(\mathcal{M})$ such that $\tilde{\mathcal{V}}_{\pi}(x(t_{f}))=v_{\pi}(t_{f})$. It is  shown in \cite{Lee2} that this  extension always exists.

Employing  (\ref{tay})  on $h$ along $v_{\pi}(t_{f})$ and using the extended smooth vector field $\tilde{\mathcal{V}}_{\pi}\in \mathfrak{X}(\mathcal{M})$, we have
\EQ \label{rrr}h(exp_{x^{o}(t_{f})}\theta v_{\pi}(t_{f}))\hspace{0cm}=\hspace{0cm}h(x^{o}(t_{f}))+\theta (\nabla_{\tilde{\mathcal{V}}_{\pi}}h)(x^{o}(t_{f}))+o(\theta),  0<\theta< \theta^{*}.\nnum\\\EN

Here we show that $K_{t_{f}}$, defined in Lemma \ref{l6}, contains all the state perturbations at $t_{f}$.
 Lemma \ref{l2} and Proposition \ref{p2} together imply that \\$K^{1}_{t_{f}}=\bigcup_{t_{s}\leq t\leq t_{f}}\bigcup_{u_{1}\in U}T\Phi^{(t_{f},t)}_{f_{q_{1}}}[f_{q_{0}}(x^{o}(t),u_{1})-f_{q_{0}}(x^{o}(t),u^{o}(t))]$ contains all the state perturbations at $x(t_{f})$ for all the elementary control perturbations inserted after   $t_{s}$. 
For all the control perturbations applied at $t_{0}<t<t_{s}$, either $t_{s}(\epsilon)<t_{s}$ or $t_{s}\leq t_{s}(\epsilon)$, where $t_{s}(\epsilon)$ is the switching time corresponding to $u_{\pi}(t,\epsilon)$.

Following Lemmas \ref{l555} and \ref{44}, in a local chart around $x(t_{s})$, the differentiability of $t_{s}(\epsilon)$ with respect to $\epsilon$ implies
\EQ\frac{d \Phi_{\pi, f_{q_{1}}}^{(t_{s}(\epsilon),t^{1}),x} }{d\epsilon}|_{\epsilon=0}&\hspace{-.1cm}=&\hspace{-.1cm}T\zeta\circ T\Phi^{(t^{-}_{s},t^{1})}_{f_{q_{0}}}[f_{q_{0}}(x^{o}(t^{1}),u_{1})-f_{q_{0}}(x^{o}(t^{1}),u^{o}(t^{1}))]\nnum\\&&\hspace{-.4cm}+(\frac{d t_{s}(\epsilon)}{d \epsilon}|_{\epsilon=0})\Big(T\zeta\big[f_{q_{0}}(x^{o}(t^{-}_{s}),u^{o}(t^{-}_{s}))\big]-f_{q_{1}}(x^{o}(t_{s}),u^{o}(t_{s}))\Big)\nnum\\&&\hspace{-.4cm}\in T_{x^{o}(t_{s})}\mathcal{M},\nnum\EN
therefore
 \EQ &&\hspace{-0cm}K^{2}_{t_{f}}= \bigcup_{t_{0}< t<t_{s}}\bigcup_{u\in U}\big\{T\Phi^{(t_{f},t_{s})}_{f_{q_{1}}}\circ T\zeta\circ T\Phi^{(t^{-}_{s},t)}_{f_{q_{0}}}[f_{q_{0}}(x^{o}(t,u_{1}))-f_{q_{0}}(x^{o}(t),u^{o}(t))] \nnum\\&&\hspace{.6cm}+(\frac{dt_{s}(\epsilon)}{d\epsilon}|_{\epsilon=0})T\Phi^{(t_{f},t_{s})}_{f_{q_{1}}}\Big(T\zeta \big[f_{q_{0}}(x^{o}(t^{-}_{s}),u^{o}(t^{-}_{s}))\big]-f_{q_{1}}(x^{o}(t_{s}),u^{o}(t_{s}))\Big)\big\}\nnum\\&&\hspace{4.6cm}\subset T_{x^{o}(t_{f})}\mathcal{M}, t\in(t_{0},t_{s}),\nnum\EN
 contains all the state variations at $x^{o}(t_{f})$ corresponding to all elementary control perturbations at $t\in(t_{0},t_{s})$. 

Since $K_{t_{f}}$ contains all the state perturbations at  $x^{o}(t_{f})$,  choosing $v_{\pi}(t_{f})\in K_{t_{f}}\subset T_{x(t_{f})}\mathcal{M}$ implies that at least at one particular time, one particular elementary control variation $\big(u_{\pi(t^{1}(v_{\pi}), u_{1}(v_{\pi}))}(t,\epsilon),$ where $u_{1}(v_{\pi})$ is the needle control resulting in the control variation $u_{\pi}(t,\epsilon)$\big) results in the final state variation  $v_{\pi}(t_{f})\in K_{t_{f}}$. 

Note that choosing $\epsilon=\theta$,  $h(exp_{x^{o}(t_{f})}\theta v_{\pi}(t_{f}))$ and $h(x_{\epsilon}(t_{f}))$, where $x_{\epsilon}(t_{f})$ is the final state  curve obtained with respect to $\epsilon$ variation, are equal to first order since they have  the same first order derivative with respect to $\epsilon$.  By the construction of $u_{\pi}(t,\epsilon)\in \mathcal{U}$, $x_{\epsilon}(t_{f})$ is a curve in the reachable set of the hybrid system at $t_{f}$. The minimality of $x^{o}(t_{f})$ consequently implies that $h(x^{o}(t_{f}))\leq h(x_{\epsilon}(t_{f}))$; then $h(x_{\epsilon}(t_{f}))- h(exp_{x^{o}(t_{f})}\epsilon v_{\pi}(t_{f}))=o(\epsilon)$ together with (\ref{rrr}) implies

\EQ\label{rr} 0\leq(\nabla_{\tilde{\mathcal{V}}_{\pi}}h)(x^{o}(t_{f})),\quad \tilde{\mathcal{V}}_{\pi}(x^{o}(t_{f}))=v_{\pi}(t_{f}), \quad\forall v_{\pi}(t_{f})\in K_{t_{f}}.\EN

For the smooth function $h:\mathcal{M}\rightarrow R$, (\ref{levi}) (ii) implies 
\EQ \tilde{\mathcal{V}}_{\pi}(h)(x^{o}(t_{f}))=\big(\nabla_{\tilde{\mathcal{V}}_{\pi}}h\big)(x^{o}(t_{f}))=\sum^{n}_{i=1}v^{i}_{\pi}(t_{f})\frac{\partial h}{\partial x_{i}}(x^{o}(t_{f})),\nnum\EN
where the second equality uses local coordinates.
Therefore by the definition of $dh$ we have 
\EQ  \big(\nabla_{\tilde{\mathcal{V}}_{\pi}}h\big)(x^{o}(t_{f}))=\langle dh(x^{o}(t_{f})),v_{\pi}(t_{f})\rangle,\nnum\EN
which implies 
\EQ\langle dh(x^{o}(t_{f})),v_{\pi}(t_{f})\rangle\geq 0, \quad\forall v_{\pi}(t_{f})\in K_{t_{f}},\nnum \EN
and completes the proof.
\end{proof}\\\\

\subsection{Statement of the Hybrid Minimum Principle}

Generalizing the results for $\mathcal{M}=\mathds{R}^{n}$ in \cite{Shaikh}, we have the following theorem which gives the HMP for  autonomous hybrid systems with only one autonomous switching which occurs on the switching manifold $\mathcal{S}\subset \mathcal{M}$.

 For an MHOCP  with a single switching from the discrete state $q_{0}$ to the discrete state $q_{1}$ at the unique switching time $t_{s}$ on the optimal trajectory $(x^{o}(.),u^{o}(.))$, where the switching manifold is an $n-1$ dimensional oriented submanifold of $\mathcal{M}$, we have 
\EQ\label{kir1} \forall X\in T_{x}\mathcal{S},\hspace{.2cm}g_{\mathcal{M}}(N_{x},X)=0,\EN
where $N_{x}\in T^{\perp}_{x}\mathcal{S}\subset T_{x}\mathcal{M}$ is the normal vector at $x^{o}(t^{-}_{s})$ (the metric $g_{\mathcal{M}}$ is positive definite). For use below we define a one form $dN_{x}$, corresponding to $N_{x}$ by
\EQ dN_{x}:=g_{\mathcal{M}}(N_{x},.)\in T^{*}_{x}\mathcal{M},\EN
where the linearity of $dN_{x}$ follows from the bi-linearity of $g_{\mathcal{M}}$.

\begin{theorem}
\label{t2}
Consider an impulsive  MHOCP  satisfying hypotheses  \textbf{\textit{A1-A3}}. Then  corresponding to  an  optimal  control and optimal state trajectory, $u^{o}$ and $x^{o}$ with a single switching state at $(x^{o}(t_{s}),t_{s})$,  there exists a nontrivial adjoint trajectory $\lambda^{o}(.)=(x^{o}(.),p^{o}(.))\in T^{*}\mathcal{M},$ defined along the optimal state trajectory, such that:
\EQ \hspace{.6cm}H_{q_{i}}(x^{o}(t),p^{o}(t),u^{o}(t))\leq H_{q_{i}}(x^{o}(t),p^{o}(t),u_{1}), \hspace{.2cm}\forall u_{1}\in U, t\in[t_{0},t_{f}],\hspace{.2cm}i=0,1,\nnum \EN
and the corresponding optimal adjoint trajectory $\lambda^{o}(.)\in T^{*}\mathcal{M}$ satisfies:
\EQ \dot{\lambda^{o}}(t)=\overrightarrow{H}_{q_{i}}(\lambda^{o}(t)),\quad t\in [t_{0},t_{f}], i=0,1,\nnum\EN
for optimal switching state and switching time $(x^{o}({t_{s}}),t_{s})$, there exists $dN_{x}\in T^{*}_{x}\mathcal{M}$ such that 

\EQ \label{kkk}&&\hspace{-.5cm}p^{o}(t^{-}_{s})=T^{*}\zeta(p^{o}(t_{s}))+\mu dN_{x^{o}(t^{-}_{s})},\nnum\\&& \hspace{-.5cm} p^{o}(t^{-}_{s})\in T^{*}_{x^{o}(t^{-}_{s})}\mathcal{M},\quad p^{o}(t_{s})\in T^{*}_{x^{o}(t_{s})}\mathcal{M},\nnum\\&&\hspace{-.5cm} x^{o}(t_{s})=\zeta(x^{o}(t^{-}_{s})),\EN
\EQ \hspace{.5cm}x^{o}(t_{0})=x_{0},\hspace{.2cm} p^{o}(t_{f})=dh(x^{o}(t_{f}))\in T^{*}_{x^{o}(t_{f})}\mathcal{M},\hspace{.2cm} dh= \sum^{n}_{i=1}\frac{\partial h}{\partial x^{i}}dx^{i}\in T^{*}_{x}\mathcal{M},\EN
where $\mu\in \mathds{R}$ and $T^{*}\zeta:T^{*}\mathcal{M}\rightarrow T^{*}\mathcal{M}$.
The continuity of the Hamiltonian at $(x^{o}({t_{s}}),t_{s})$ is given as follows
\EQ\hspace{-.5cm} H_{q_{0}}(x^{o}(t^{-}_{s}),p^{o}(t^{-}_{s}),u^{o}(t^{-}_{s}))=H_{q_{1}}(x^{o}(t_{s}),p^{o}(t_{s}),u^{o}(t_{s})).\nnum\EN

\end{theorem}
\begin{proof}
The proof is based on the control needle variation along the optimal state trajectory and employs the results of Lemma \ref{l6}; it is given in Appendix \ref{s1}.
\end{proof}\\
In the case where  $dim(\mathcal{S})<n-1$, the normal vector at the optimal switching state is not uniquely defined and (\ref{kkk}) becomes
\EQ\hspace{.8cm}p^{o}(t^{-}_{s})-T^{*}\zeta(p^{o}(t_{s}))\in T^{*^{\perp}}_{x^{0}(t^{-}_{s})}\mathcal{S}\hspace{.5cm} p^{o}(t^{-}_{s})\in T^{*}_{x^{o}(t^{-}_{s})}\mathcal{M},\quad p^{o}(t_{s})\in T^{*}_{x^{o}(t_{s})}\mathcal{M},\nnum\EN
where $T^{*^{\perp}}_{x}\mathcal{S}:=\{\alpha\in T^{*}_{x}\mathcal{M},\hspace{.3cm} s.t. \hspace{.2cm} \forall X\in T_{x}\mathcal{S},\langle\alpha,X\rangle=0\}$.
\\\\
\section{Simulation Results}
To illustrate the results above we consider an HOCP and employ the \textit{Gradient Geodesic-HMP} (GG-HMP) algorithm (see \cite{Taringoo4}). 

The HOCP is defined on a torus with the following parametrization:
\EQ &&\hspace{-.5cm}x(\zeta,w)=(R+rcos(w))cos(\zeta),\nnum\\&&\hspace{-.5cm} y(\zeta,w)=(R+rcos(w))sin(\zeta),\nnum\\ &&\hspace{-.5cm}z(\zeta,w)=rsin(w), w,\zeta\in [0,2\pi).\nnum\EN
where $R=1, r=0.5$. The induced Riemannian metric is given by
\EQ g_{T^{2}}(\zeta, w)=(R+rcos(w))^{2}d\zeta \otimes d\zeta+r^{2}dw \otimes dw.\nnum\EN
\begin{figure}
\begin{center}
\hspace*{-4cm}\includegraphics[scale=.51]{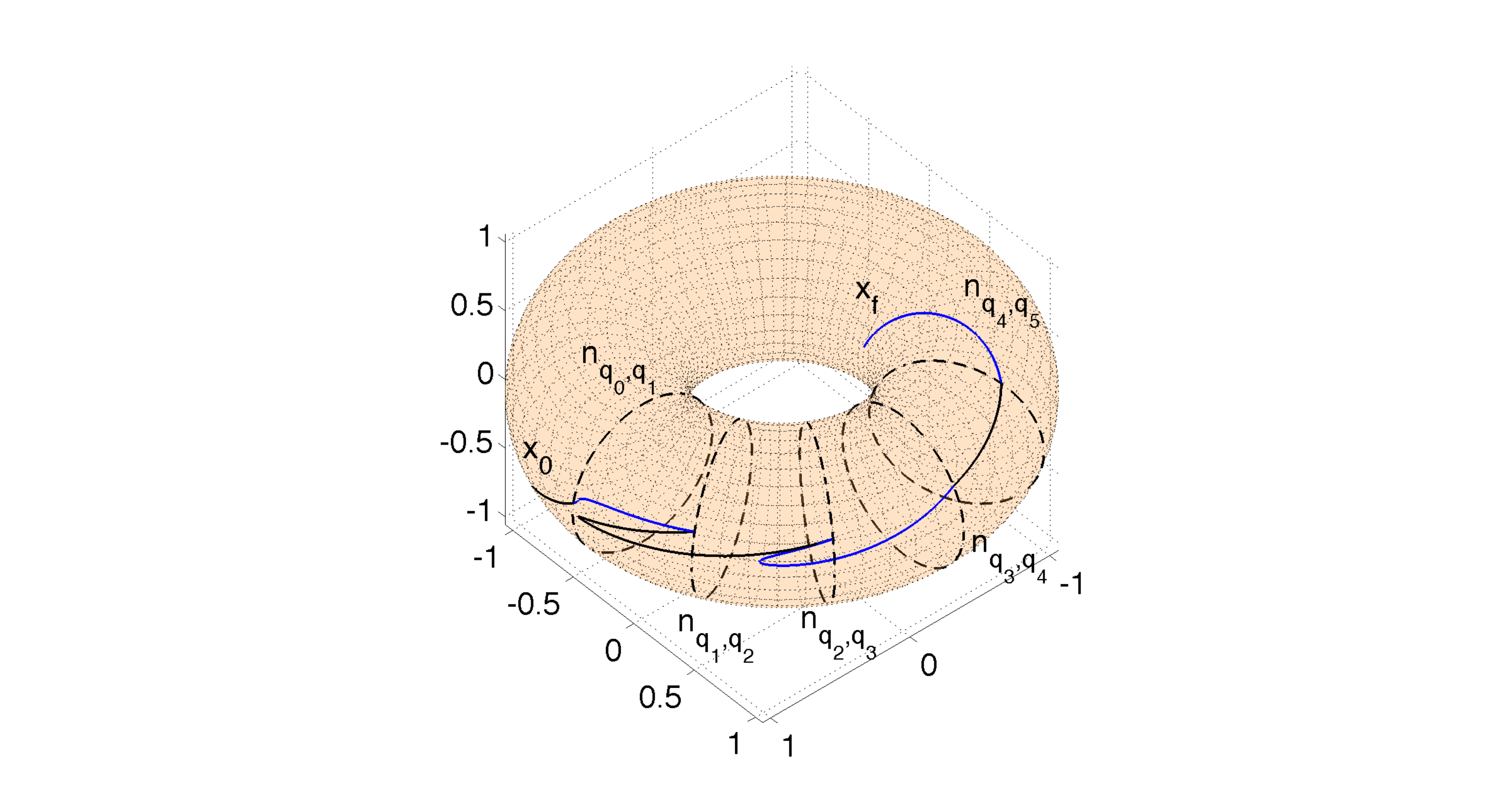}
\vspace{-1cm}
      \caption{ Hybrid State Trajectory On the Torus}
      \label{1222}
      \end{center}
   \end{figure}
\begin{figure}
\begin{center}
\hspace*{-1.25cm}\includegraphics[scale=.43]{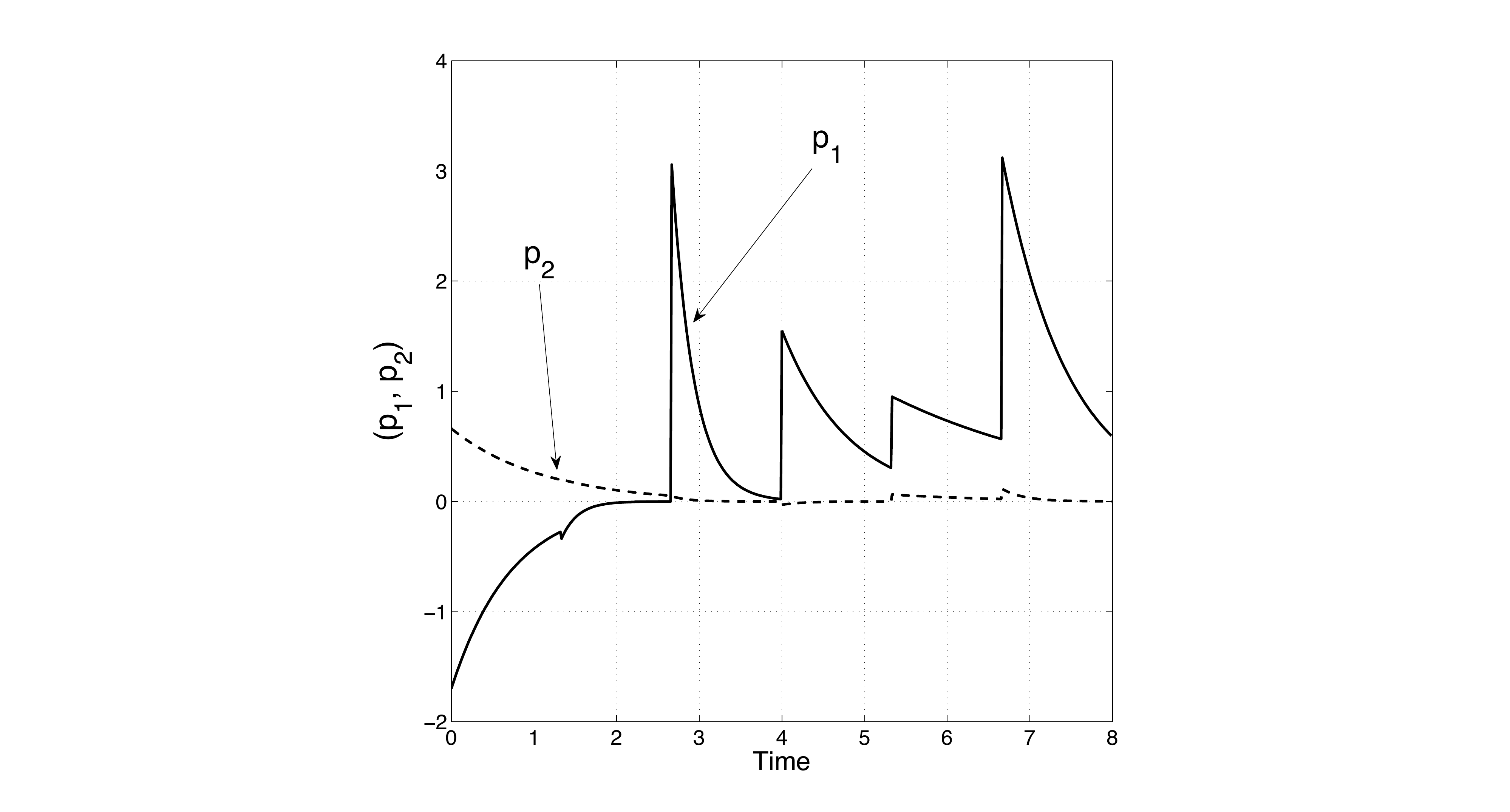}
      \caption{ Adjoint Processes }
      \label{555}
      \end{center}
   \end{figure}
\begin{figure}
\begin{center}
\hspace*{3cm}\includegraphics[scale=.43]{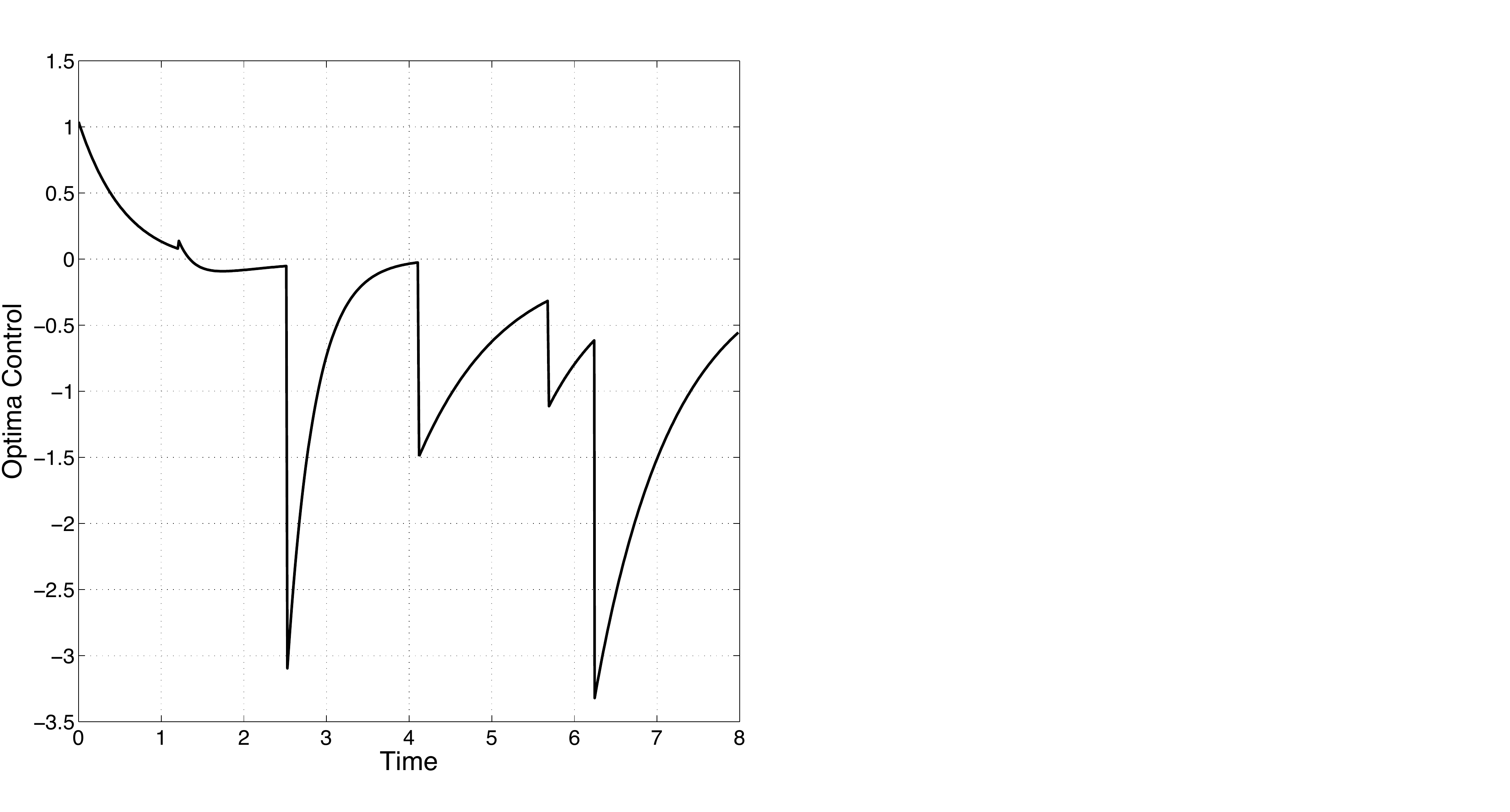}
      \caption{ Control Function}
      \label{5555}
      \end{center}
   \end{figure}

The hybrid system trajectory goes through each discrete state in numerical order and the dynamics are given in the local parametrization space of the torus $T^{2}$ as follows: 
 \EQ\hspace{.3cm} \textsl{q}_{0}\quad   \left(\begin{array}{cc}\dot{\zeta}\\\dot{w}\end{array}\right)=\left(\begin{array}{cc}1.5\quad 0\\0 \quad 1\end{array}\right)\left(\begin{array}{cc} \zeta\\ w \end{array}\right)+\left(\begin{array}{cc}1\\1\end{array}\right)u,\EN
\EQ \textsl{q}_{1}\quad   \left(\begin{array}{cc}\dot{\zeta}\\\dot{w}\end{array}\right)=\left(\begin{array}{cc}5\quad 0\\0 \quad 1\end{array}\right)\left(\begin{array}{cc}\zeta\\w \end{array}\right)+\left(\begin{array}{cc}1\\1\end{array}\right)u,\EN
\EQ \textsl{q}_{2}\quad   \left(\begin{array}{cc}\dot{\zeta}\\\dot{w}\end{array}\right)=\left(\begin{array}{cc}3\quad 0\\0 \quad 4\end{array}\right)\left(\begin{array}{cc}\zeta\\w \end{array}\right)+\left(\begin{array}{cc}1\\1\end{array}\right)u,\EN
\EQ \textsl{q}_{3}\quad   \left(\begin{array}{cc}\dot{\zeta}\\\dot{w}\end{array}\right)=\left(\begin{array}{cc}1\quad 0\\0 \quad 3\end{array}\right)\left(\begin{array}{cc}\zeta\\w \end{array}\right)+\left(\begin{array}{cc}1\\1\end{array}\right)u,\EN
\EQ \textsl{q}_{4}\quad   \left(\begin{array}{cc}\dot{\zeta}\\\dot{w}\end{array}\right)=\left(\begin{array}{cc}1\quad 0\\0 \quad 2\end{array}\right)\left(\begin{array}{cc}\zeta\\w \end{array}\right)+\left(\begin{array}{cc}1\\1\end{array}\right)u,\EN
\EQ \textsl{q}_{5}\quad   \left(\begin{array}{cc}\dot{\zeta}\\\dot{w}\end{array}\right)=\left(\begin{array}{cc}1\quad 0\\0 \quad 3\end{array}\right)\left(\begin{array}{cc}\zeta\\w \end{array}\right)+\left(\begin{array}{cc}1\\1\end{array}\right)u.\EN
The switching submanifolds and the cost function are defined as follows:
\EQ n_{q_{0},q_{1}}=\{0\leq w<2\pi, \zeta=0\},\hspace{,5cm} n_{q_{1},q_{2}}=\{0\leq w<2\pi, \zeta=\frac{\pi}{6}\},\EN
\EQ n_{q_{2},q_{3}}=\{0\leq w<2\pi, \zeta=\frac{\pi}{3}\},\hspace{.5cm} n_{q_{3},q_{4}}=\{0\leq w<2\pi, \zeta=\frac{\pi}{2}\},\EN
\EQ n_{q_{4},q_{5}}=\{0\leq w<2\pi, \zeta =\frac{2\pi}{3}\}, \hspace{.5cm}J=\frac{1}{2}\int^{8}_{0}u^{2}(t)dt,\EN
and the boundary conditions are given as:
\EQ x_{0}=(1.4117, -0.4367, -0.1478)\in \mathds{R}^{3},\hspace{.1cm}\\x_{f}=(-0.1478, -0.49980, 0.10130)\in \mathds{R}^{3}.\nnum\EN
The hamiltonian functions are given as 
\EQ H_{q_{i}}(\left(\begin{array}{cc}\zeta\\w\end{array}\right),p(t),u(t))=(p_{1}(t),p_{2}(t)) \left(\begin{array}{cc}\dot{\zeta}\\\dot{w}\end{array}\right)+\frac{1}{2}u^{2}(t),\hspace{.2cm} i=0,...,5,\hspace{.2cm} t\in[t_{i},t_{i+1}).\nnum\\\EN
The GG-HMP algorithm is an extension to Riemannian manifolds of the HMP algorithm introduced in \cite{Shaikh};
 this is done by introducing  a geodesic gradient flow algorithm on $\mathcal{S}$ and constructing an HMP algorithm   along geodesics on $\mathcal{S}$.
Figure \ref{1222} shows the state trajectory on the torus and Figure \ref{555} depicts the adjoint variable with the discontinuity at the optimal switching times \\$ \overline{t_{s}}=[ 1.2159, 2.5142, 4.0571, 5.6186, 6.3307] $.

\appendix
 \section{Proof of Lemma \ref{44}}
 \label{A0}
 \begin{proof}
 Since $\mathcal{S}$ is a smooth embedded submanifold of $\mathcal{M}$ the inclusion $i:\mathcal{S}\rightarrow \mathcal{M}$ is a topological embedding and hence its rank is constant (see \cite{Lee2}). By the Rank Theorem for Manifolds (see \cite{Lee3}), $i$ may be locally given as
 \EQ i(x_{1},...,x_{n-1})=(x_{1},...,x_{n-1},0).\nnum \EN 
 Hence, $\mathcal{S}$ is locally homeomorphic to $\mathds{R}^{n-1}$. As stated in Subsection \ref{ss1}, $\Phi^{(t_{s},t_{0})}_{\pi,f_{q_{0}}}(x_{0})$ converges to $x^{o}(t_{s})\in \mathcal{S}$ as $\epsilon\downarrow 0$ (see \cite{Piccoli}), therefore $\Phi^{(t_{s},t_{0})}_{\pi,f_{q_{0}}}(x_{0})$ converges into any neighbourhood of $x^{o}(t_{s})\in\mathcal{S}$ as $\epsilon\downarrow 0$. Let us denote the coordinate domain neighbourhood given by the Rank Theorem as $U_{x^{o}(t^{-}_{s})}$, where $x^{o}(t^{-}_{s})=\Phi^{(t^{-}_{s},t^{1})}_{f_{q_{0}}}(x(t^{1}))\in \mathcal{S}$. \\
  
 Consider $0<\delta t$ such that $\Phi^{(t_{s}+\delta t,t_{s})}_{f_{q_{0}}}(x^{o}(t^{-}_{s}))\in U_{x^{o}(t^{-}_{s})}$. In the local coordinate system around $x^{o}(t^{-}_{s})$ defined above, the switching manifold $\mathcal{S}$ separates $U_{x^{o}(t^{-}_{s})}$ into two subsets $U^{1}_{x^{o}(t^{-}_{s})}, U^{2}_{x^{o}(t^{-}_{s})}$, where $U^{1}_{x^{o}(t^{-}_{s})}=\{x\in U_{x^{o}(t^{-}_{s})}, x_{n}<0\}$ and $U^{2}_{x^{o}(t^{-}_{s})}=\{x\in U_{x^{o}(t^{-}_{s})}, x_{n}>0\}$. For definiteness, we assume that first, the state trajectory enters $U^{1}_{x^{o}(t^{-}_{s})}$ and second, it enters $U^{2}_{x^{o}(t^{-}_{s})}$ after meeting the switching manifold; therefore $\Phi^{(t_{s}+\delta t,t_{s})}_{f_{q_{0}}}(x^{o}(t^{-}_{s}))\in U^{2}_{x^{o}(t^{-}_{s})}$ for all sufficiently small $\delta t>0$.   The convergence of $\Phi^{(t_{s}+\delta t,t_{s})}_{\pi,f_{q_{0}}}(x(t_{s}))$ to $\Phi^{(t_{s}+\delta t,t_{s})}_{f_{q_{0}}}(x^{o}(t^{-}_{s}))$ implies that for sufficiently small $\epsilon$, $\Phi^{(t_{s}+\delta t,t_{s})}_{\pi,f_{q_{0}}}(x(t_{s}))\in U^{2}_{x^{o}(t^{-}_{s})}$, hence by the continuity of the trajectory there exists a switching time, $t_{s}(\epsilon)$, such that $\Phi^{(t_{s}(\epsilon),t_{s})}_{\pi, f_{q_{0}}}(x(t_{s}))\in \mathcal{S}$, see Figure \ref{1111}.
 \begin{figure}
 \vspace{0cm}
\begin{center}
\hspace*{0cm}\includegraphics[scale=.7]{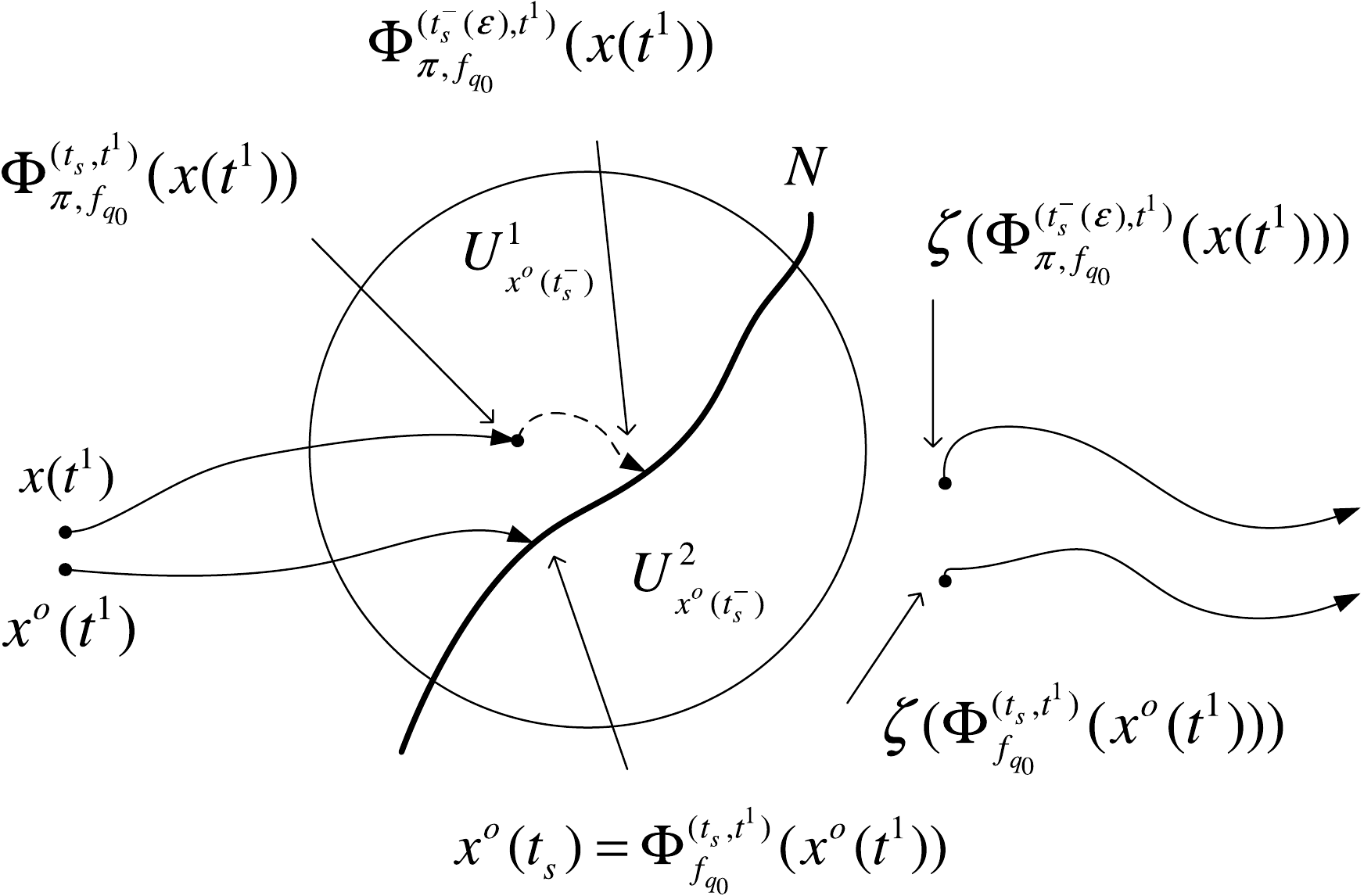}
\vspace{0cm}
      \caption{ Nominal and Perturbed State Trajectories}
      \label{1111}
      \end{center}
   \end{figure}
   \\
Furthermore by the continuity of the  state trajectory, we may choose $0\leq \epsilon$ sufficiently small that $\Phi^{(t_{s}(\epsilon),t_{s})}_{\pi, f_{q_{0}}}(x(t_{s}))\in U_{x^{o}(t^{-}_{s})}$. Let us define $\Psi:\mathds{R}^{+}\times \mathds{R}\rightarrow \mathds{R}$  by $\Psi(\epsilon,t)=x_{n}\circ\Phi^{(t,t^{1})}_{\pi, f_{q_{0}}}(x(t^{1}))$, where $x_{n}$ is the last coordinate function. Hence, the differentiability of $\Psi$ with respect to $t$ is immediate by the construction of $\Psi$ since $\frac{d\Psi(\epsilon,t)}{d t}|_{t_{s}(\epsilon)}=f^{n}_{q_{0}}(x(t_{s}(\epsilon)))$, where $f^{n}_{q_{0}}$ is the corresponding coefficient of the last coordinate of $f_{q_{0}}$. 
\\
In order to show the differentiability of $\Psi$ with respect to $\epsilon$ the following needle variation is applied
 \EQ \label{u1} u_{\pi}(t,\epsilon)=\left\{ \begin{array}{cc} \quad \hspace{-1cm}u^{o}(t) \quad \hspace{.2cm}t\leq t^{1}-\epsilon\\ \hspace{.2cm}u_{1}\quad  \hspace{.6cm}t^{1}-\epsilon\leq t\leq t^{1}\\ \hspace{-.5cm}u^{o}(t) \quad \hspace{.2cm}t^{1}< t \leq t_{s}\\u^{o}(t_{s})\quad t_{s}\leq t < t_{s}(\epsilon) \end{array}\right.,\EN

From Section 3 we recall that $\Phi_{\pi, f_{q}}^{(t, s), x}(\epsilon):=\Phi_{f_{q}^{u_{\pi}(t,\epsilon)}}^{(t, s)}(x(s))$ then one can verify that, by the results of Lemma \ref{l2} and Proposition \ref{p2}, the needle variation control $ u_{\pi}(t,\epsilon)$, given in (\ref{needle}), results in the following tangent perturbation vector at $t^{1}$, where $\epsilon \in[0,\epsilon_{0})$ for some $\epsilon_{0}>0$.
\EQ \label{12222}\frac{d}{d\epsilon}\Phi_{\pi, f_{q_{0}}}^{(t^{1},s),x}|_{\epsilon}&=&\lim_{\delta\rightarrow 0}\frac{\Phi_{\pi, f_{q_{0}}}^{(t^{1},s),x}(\epsilon+\delta)-\Phi_{\pi, f_{q_{0}}}^{(t^{1},s),x}(\epsilon)}{\delta}\nnum\\&=&T\Phi^{(t^{1},t^{1}-\epsilon)}_{\pi, f_{q_{0}}}\Big(f_{q_{0}}(x(t^{1}-\epsilon),u_{1})-f_{q_{0}}(x(t^{1}-\epsilon),u(t^{1}-\epsilon))\Big)\nnum\\&&\in T_{x(t^{1})}\mathcal{M}=T_{\Phi_{\pi, f_{q_{0}}}^{(t^{1},s),x}(\epsilon)}\mathcal{M}.\nnum\\\EN 
That implies the differentiability of $\Psi$ on $[0,\epsilon_{0})$, where (see Figure \ref{1112})
\EQ \frac{d}{d\epsilon}\Phi_{\pi, f_{q_{0}}}^{(t,t^{1}),x}|_{\epsilon}= T\Phi^{(t,t^{1})}_{ f_{q_{0}}}(\frac{d}{d\epsilon}\Phi_{\pi, f_{q_{0}}}^{(t^{1},s),x}|_{\epsilon}), \hspace{.5cm}t\in [t^{1},t^{-}_{s}(\epsilon)).\nnum\EN

The transversality hypothesis at the intersection of the state trajectory and the switching manifold implies that  $f^{n}_{q_{0}}(x(t_{s}(\epsilon)))\ne 0$; then by employing the Implicit Function Theorem (see \cite{Lee2}, Theorem 7.9) we have 
\EQ \Psi(\epsilon,t_{s}(\epsilon))=0\Rightarrow\exists \kappa: \mathds{R}\rightarrow \mathds{R}, \hspace{.5cm}s.t. \quad \kappa(\epsilon)=t_{s}(\epsilon),\nnum\EN
and $\kappa$ and $\Psi$ both are $C^{1}$; then the derivative of $\kappa(.)$ with respect to $\epsilon$ is given as
\EQ\label{kap}\frac{d\kappa(\epsilon)}{d\epsilon}=-(\frac{\partial \Psi}{\partial t})^{-1}|_{t=t_{s}(\epsilon)}.\frac{\partial \Psi}{\partial\epsilon}=-f^{n^{-1}}_{q_{0}}(x(t_{s}(\epsilon))).T^{n}\Phi^{(t_{s}(\epsilon),t^{1})}_{ f_{q_{0}}}(\frac{d}{d\epsilon}\Phi_{\pi, f_{q_{0}}}^{(t^{1},s),x}|_{\epsilon}),\nnum\\\EN
where $T^{n}\Phi^{(t_{s}(\epsilon),t^{1})}_{ f_{q_{0}}}(\frac{d}{d\epsilon}\Phi_{\pi, f_{q_{0},u_{1}}}^{(t^{1},s),x}|_{\epsilon})$ is the $n$th coordinate of $T\Phi^{(t_{s}(\epsilon),t^{1})}_{ f_{q_{0}}}(\frac{d}{d\epsilon}\Phi_{\pi, f_{q_{0}}}^{(t^{1},s),x}|_{\epsilon})$. 

This completes the proof of differentiability of $t_{s}(\epsilon)$ with respect to $\epsilon$.
 The proof for the differentiability of $t_{s}(\epsilon)$ in the case where $t_{s}(\epsilon)<t_{s}$ parallels the proof given above.
 \begin{figure}
 \begin{center}
\includegraphics[scale=.7]{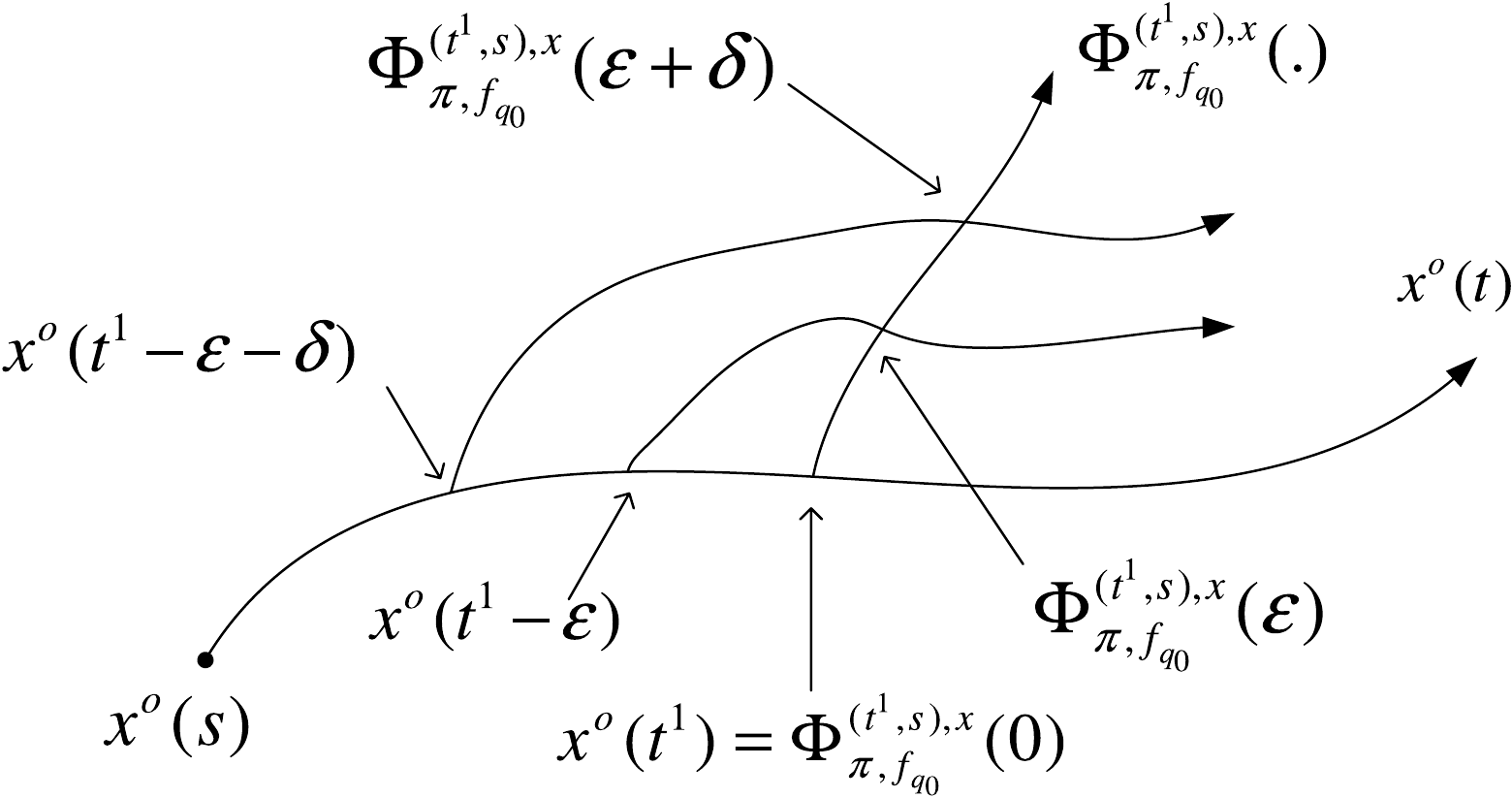}
      \caption{ Nominal and Perturbed State Trajectories}
      \label{1112}
      \end{center}
   \end{figure}
  \end{proof}
  
     \section{Proof of Lemma \ref{l555}}
\label{A}
\begin{proof}

Without loss of generality assume $t_{s}\leq t_{s}(\epsilon)$, then 

\EQ\Phi_{\pi, f_{q_{1}}}^{(t_{s}(\epsilon),t^{1}),x}(\epsilon)=\zeta\circ  \Phi_{\pi, f_{q_{0}}}^{(t^{-}_{s}(\epsilon),t^{1}),x}(\epsilon), \hspace{1cm }t^{1}\in [t_{0},t_{s}),\nnum\EN
where $\Phi_{\pi, f_{q_{0}}}^{(t^{-}_{s}(\epsilon),t^{1}),x}(\epsilon)=\Phi_{\pi, f_{q_{0}}}^{(t^{-}_{s}(\epsilon),t_{s})}\circ \Phi_{\pi, f_{q_{0}}}^{(t^{-}_{s},t^{1}),x}(\epsilon)$ and $x(t^{1})=x$, then in a local coordinate system of $x(t_{s})$ we have

\EQ &&\hspace{0cm}\zeta\circ \Phi_{\pi, f_{q_{0}}}^{(t^{-}_{s}(\epsilon),t_{s})}\circ\Phi_{\pi, f_{q_{0}}}^{(t^{-}_{s},t^{1}),x}(\epsilon) -\Phi_{f_{q_{1}}}^{(t_{s}(\epsilon),t_{s})}\circ \zeta\circ \Phi_{ f_{q_{0}}}^{(t^{-}_{s},t^{1}),x}(\epsilon)=\nnum\\&&\hspace{0cm}\Big\{\zeta\big(\int^{t_{s}(\epsilon)}_{t_{s}}f_{q_{0}}(x_{\epsilon}(t),u^{o}(t_{s}))dt+x_{\epsilon}(t_{s})\big)\Big\}-\Big\{\int^{t_{s}(\epsilon)}_{t_{s}}f_{q_{1}}(x^{o}(t),u^{o}(t))dt+\zeta(x^{o}(t^{-}_{s}))\Big\}.\nnum\EN

Since $u_{\pi}(t,\epsilon)=u^{o}(t_{s}), t\in[t_{s},t^{-}_{s}(\epsilon))$, $f_{q_{0}}(x_{\epsilon}(t),u^{o}(t_{s}))$ is differentiable with respect to $t$. Hence by the Taylor expansion of $\zeta$ around $x_{\epsilon}(t_{s})$ and the Mean Value Theorem we have
 \EQ&&\Big\{\zeta\big(\int^{t^{-}_{s}(\epsilon)}_{t_{s}}f_{q_{0}}(x_{\epsilon}(t),u^{o}(t_{s}))dt+x_{\epsilon}(t_{s})\big)\Big\}=\nnum\\&&\zeta(x_{\epsilon}(t_{s}))+(t_{s}(\epsilon)-t_{s})\times T\zeta.f_{q_{0}}(x_{\epsilon}(\hat{t}),u^{o}(t_{s}))+o(\delta x),\nnum \EN
 where $\hat{t}\in(t_{s},t^{-}_{s}(\epsilon))$. Applying the Taylor expansion of $\zeta$ around $x^{o}(t_{s})$ imples
 \EQ \zeta(x_{\epsilon}(t_{s}))-\zeta(x^{o}(t^{-}_{s}))=T\zeta (x_{\epsilon}(t_{s})-x^{o}(t^{-}_{s}))+o(\delta x),\nnum\EN
 where by the definition of the derivatives we have 
 \EQ\label{pp} \frac{d \Phi_{\pi, f_{q_{1}}}^{(t_{s}(\epsilon),t^{1}),x} }{d\epsilon}|_{\epsilon=0}&=&\lim_{\epsilon\downarrow 0}\frac{\Phi_{\pi, f_{q_{1}}}^{(t_{s}(\epsilon),t^{1}),x}-\Phi_{f_{q_{1}}}^{(t_{s}(\epsilon),t^{1}),x}}{\epsilon}\nnum\\&=&\lim_{\epsilon\downarrow 0}\frac{\zeta\circ \Phi_{\pi, f_{q_{0}}}^{(t^{-}_{s}(\epsilon),t_{s})}\circ\Phi_{\pi, f_{q_{0}}}^{(t^{-}_{s},t^{1}),x} -\Phi_{f_{q_{1}}}^{(t_{s}(\epsilon),t_{s})}\circ \zeta\circ \Phi_{ f_{q_{0}}}^{(t^{-}_{s},t^{1}),x}}{\epsilon},\EN
 therefore as $\epsilon\downarrow 0$, Lemma \ref{44} and (\ref{pp}) together yield
\EQ  \frac{d \Phi_{\pi, f_{q_{1}}}^{(t_{s}(\epsilon),t^{1}),x} }{d\epsilon}|_{\epsilon=0}&\hspace{-.1cm}=&\hspace{-.1cm}T\zeta\circ T\Phi^{(t^{-}_{s},t^{1})}_{f_{q_{0}}}\circ \frac{d \Phi_{\pi, f_{q_{0}}}^{(t^{1},t_{0}),x_{0}} }{d \epsilon}|_{\epsilon=0}\nnum\\&&+\frac{d t_{s}(\epsilon)}{d \epsilon}|_{\epsilon=0}.\big(T\zeta\big(f_{q_{0}}(x^{o}(t^{-}_{s}),u^{o}(t^{-}_{s}))\big)-f_{q_{1}}(x^{o}(t_{s}),u^{o}(t_{s}))\big).\nnum \EN
Lemma \ref{l2} and Proposition \ref{p2} complete the proof for the case $t_{s}\leq t_{s}(\epsilon)$. The same argument holds for $t_{s}(\epsilon)< t_{s} $ with a sign change for $\frac{d t_{s}(\epsilon)}{d \epsilon}|_{\epsilon=0}$.
It should be noted that the derivative in (\ref{pp}) gives the state variation at $t_{s}(\epsilon)$, therefore the nominal flow is subtracted from the perturbed one up to $t_{s}(\epsilon)$.\end{proof}

 \section{Proof of  Theorem \ref{t2}}\label{s1}
 \begin{proof}
We split the proof into the following steps:
First, the needle variation is applied at $t^{1}$, where $t_{s}<t^{1}\leq t_{f}$ and $t_{s}$ is the optimal switching time on the switching manifold from $q_{0}$ to $q_{1}$, hence there is no switching phenomena after $t^{1}$. At this stage the proof is same as  the proof presented in \cite{Agra} and  \cite{Barbero}. Second, the needle variation is applied at $t^{1}$, $t_{0}\leq t^{1}< t_{s}$.  Third, we show that the constructed adjoint variable, $\lambda$, satisfies the Hamiltonian equations and, fourth, the continuity of the Hamiltonian at the optimal switching state $x^{o}(t_{s})$ and time $t_{s}$ is obtained.

\textit{Step 1}:
Choose the following control needle variation:
\EQ u_{\pi}(t,\epsilon)=\left\{ \begin{array}{cc} \quad u_{1} \quad t^{1}-\epsilon \leq t\leq t^{1}\\ u^{o}(t)\quad\mbox{elsewhere} \end{array}\right.,\nnum  \EN
where $t_{s}< t^{1}\leq t_{f}, u_{1}\in U$.  By Lemma \ref{l2} the state variation at $t^{1}$ is \\$[f_{q_{1}}(x^{o}(t^{1}),u_{1})-f_{q_{1}}(x^{o}(t^{1}),u^{o}(t^{1}))]\in T_{x^{o}(t^{1})}\mathcal{M}$. 
By the definition of $K_{t_{f}}$ we have 
\EQ T\Phi^{(t_{f},t^{1})}_{f_{q_{1}}}([f_{q_{1}}(x^{o}(t^{1}),u_{1})-f_{q_{1}}(x^{o}(t^{1}),u^{o}(t^{1}))])\in K^{1}_{t_{f}}\subset T_{x^{o}(t_{f})}\mathcal{M}.\nnum \EN
Lemma \ref{l6} implies that
\EQ \hspace{-.8cm}0\leq\langle dh(x^{o}(t_{f})),T\Phi^{(t_{f},t^{1})}_{f_{q_{1}}}([f_{q_{1}}(x^{o}(t^{1}),u_{1})-f_{q_{1}}(x^{o}(t^{1}),u^{o}(t^{1}))])\rangle\hspace{.1cm},\nnum\EN
and by  Proposition \ref{p1} 
\EQ &&\hspace{-.8cm}0\leq  \langle dh(x^{o}(t_{f})),T\Phi^{(t_{f},t^{1})}_{f_{q_{1}}}f_{q_{1}}(x^{o}(t^{1}),u_{1})-f_{q_{1}}(x^{o}(t^{1}),u^{o}(t^{1}))\rangle\nnum\\&&\hspace{-.5cm}=\langle T^{*}\Phi^{(t_{f},t^{1})}_{f_{q_{1}}}dh(x^{o}(t_{f})),f_{q_{1}}(x^{o}(t^{1}),u_{1})-f_{q_{1}}(x^{o}(t^{1}),u^{o}(t^{1}))\rangle,\nnum\\&&\hspace{4.5cm}t_{s}< t^{1}< t_{f}.\nnum \EN
Therefore
\EQ \label{lam2}&&\hspace{-.8cm}\langle T^{*}\Phi^{(t_{f},t^{1})}_{f_{q_{1}}}dh(x^{o}(t_{f})),f_{q_{1}}(x^{o}(t^{1}),u^{o}(t^{1})\rangle\hspace{.5cm}\nnum\\&&\hspace{-.8cm}\leq \langle T^{*}\Phi^{(t_{f},t^{1})}_{f_{q_{1}}}dh(x^{o}(t_{f})),f_{q_{1}}(x^{o}(t^{1}),u_{1}))\rangle,\quad t_{s}< t^{1}< t_{f},\EN
for all $u_{1}\in U$ and setting  $p^{o}(t):=T^{*}\Phi^{(t_{f},t)}_{f_{q_{1}}}dh(x^{o}(t_{f}))$ yields a trajectory $p^{o}(.)$ satisfying  the minimization statement of the theorem. 

\textit{Step 2}:
Here we use the needle variation before the optimal switching time $t_{s}$ i.e:
\EQ u_{\pi}(t,\epsilon)=\left\{ \begin{array}{cc} \quad u_{1} \quad t^{1}-\epsilon \leq t\leq t^{1}\\ u^{o}(t)\quad\mbox{elsewhere} \end{array}\right., \nnum \EN
where $t^{1}< t_{s}, u_{1}\in U$.
Similar to the first step, the derivative of the state trajectory with respect to $\epsilon$ at $t^{1}$ is obtained as $[f_{q_{0}}(x^{o}(t^{1}),u_{1})-f_{q_{0}}(x^{o}(t^{1}),u^{o}(t^{1}))]\in T_{x^{o}(t^{1})}\mathcal{M}$ and $\frac{d}{d\epsilon}\Phi_{\pi, f_{q_{0}}}^{(t,s),x}|_{\epsilon=0}=\Psi(t), t\in[t^{1},t_{s})$.
In order to use the  method introduced in the first step, we describe the evolution  of the perturbed state, $\Phi_{\pi, f_{q}}^{(t,s),x}$,  after the switching time. Note that  each elementary control variation, $u_{\pi}(t,\epsilon)$,  results in a different switching time $t_{s}$ which depends upon both of $\epsilon$ and $u_{1}$. Now let us consider a state mapping from $x(t^{1})$ to the switching state $x(t^{-}_{s}(\epsilon))$ induced by the needle control variation; then  the state variation at the optimal switching state $x^{o}(t^{-}_{s})$ is obtained as the push forward of
\EQ  \Phi_{\pi, f_{q_{0}}}^{(t^{-}_{s}(.),t^{1}),x}:[0,\tau]\rightarrow \mathcal{S},\quad x\in \mathcal{M},\quad  x(t_{s}(\epsilon))\in \mathcal{S},\nnum\EN
where $ \Phi_{\pi, f_{q_{0}}}^{(t^{-}_{s}(\epsilon),t^{1}),x}:= \Phi_{\pi, f_{q_{0}}}^{(t^{-}_{s}(\epsilon),t^{1})}(x(t^{1}))\in \mathcal{S}$ and $t_{s}(\epsilon)$ is the switching time corresponding to the selected $\epsilon$. Here we have two possibilities, $(i)$: $t_{s}\leq  t_{s}(\epsilon)$ and $(ii)$: $t_{s}(\epsilon)< t_{s}$.
The corresponding control needle variations for these two possibilities are given as follows:
\EQ (i):t_{s}\leq  t_{s}(\epsilon),\quad u_{\pi}(t,\epsilon)=\left\{ \begin{array}{cc} \quad \hspace{-1cm}u^{o}(t) \quad \hspace{.2cm}t\leq t^{1}-\epsilon\\ \hspace{.2cm}u_{1}\quad  \hspace{.6cm}t^{1}-\epsilon\leq t\leq t^{1}\\ \hspace{-.5cm}u^{o}(t) \quad \hspace{.2cm}t^{1}< t \leq t_{s}\\u^{o}(t_{s})\quad t_{s}\leq t < t_{s}(\epsilon) \end{array}\right.,\nnum\EN
and
\EQ (ii):t_{s}(\epsilon)< t_{s},\quad u_{\pi}(t,\epsilon)=\left\{ \begin{array}{cc} \quad \hspace{-1cm}u^{o}(t) \quad \hspace{.2cm}t\leq t^{1}-\epsilon\\ \hspace{.2cm}u_{1}\quad  \hspace{.6cm}t^{1}-\epsilon\leq t\leq t^{1}\\ \hspace{-.0cm}u^{o}(t) \quad \hspace{.3cm}t^{1}< t < t_{s}(\epsilon)\\u^{o}(t_{s})\quad t_{s}(\epsilon)\leq t \leq t_{s} \end{array}\right..\nnum\EN
Notice that $u^{o}(t_{s})$ in $(i)$ corresponds to $f_{q_{0}}$ under the optimal control and in $(ii)$ corresponds to $f_{q_{1}}$ under the optimal control. 
The right differentiability of $t_{s}(\epsilon)$ with respect to $\epsilon$ at $0$ by Lemma \ref{44} (since the needle variation is defined for $0\leq\epsilon$) and Lemma \ref{l555}, in case $(i)$, together imply
\EQ\label{lllq} \frac{d \Phi_{\pi, f_{q_{0}}}^{(t^{-}_{s}(\epsilon),t^{1}),x} }{d\epsilon}|_{\epsilon=0}&\hspace{-.1cm}=&\hspace{-.1cm}\big(\frac{d t_{s}(\epsilon)}{d \epsilon}|_{\epsilon=0}\big).f_{q_{0}}(x^{o}(t^{-}_{s}),u^{o}(t^{-}_{s}))\nnum\\&&\hspace{-1cm}+T\Phi^{(t^{-}_{s},t^{1})}_{f_{q_{0}}}[f_{q_{0}}(x^{o}(t^{1}),u_{1})-f_{q_{0}}(x^{o}(t^{1}),u^{o}(t^{1}))]\in T_{x^{o}(t^{-}_{s})}\mathcal{S}\subset  T_{x^{o}(t^{-}_{s})}\mathcal{M}.\nnum\\\EN
And in case $(ii)$ we have
\\\\ 
\EQ \label{lllqq} \frac{d \Phi_{\pi, f_{q_{0}}}^{(t^{-}_{s}(\epsilon),t^{1}),x} }{d\epsilon}|_{\epsilon=0}&\hspace{-.1cm}=&\hspace{-.1cm}-\big(\frac{d t_{s}(\epsilon)}{d \epsilon}|_{\epsilon=0}\big).f_{q_{0}}(x^{o}(t^{-}_{s}),u^{o}(t^{-}_{s}))\nnum\\&&\hspace{-1cm}+T\Phi^{(t^{-}_{s},t^{1})}_{f_{q_{0}}}[f_{q_{0}}(x^{o}(t^{1}),u_{1})-f_{q_{0}}(x^{o}(t^{1}),u^{o}(t^{1}))]\in T_{x^{o}(t^{-}_{s})}\mathcal{S}\subset T_{x^{o}(t^{-}_{s})}\mathcal{M}.\nnum\\\EN
In the first case, (\ref{kir1}) and  (\ref{lllq}) together yield
\EQ \label{llas}\frac{d t_{s}(\epsilon)}{d \epsilon}|_{\epsilon=0}&\hspace{-.1cm}=&\hspace{-.1cm}-\langle dN_{x^{o}(t^{-}_{s})},f_{q_{0}}(x^{o}(t^{-}_{s}),u^{o}(t^{-}_{s}))\rangle^{-1}\nnum\\&&\hspace{-.4cm}\times\langle dN_{x^{o}(t^{-}_{s})},T\Phi^{(t_{s},t^{1})}_{f_{q_{0}}}[f_{q_{0}}(x^{o}(t^{1}),u_{1})-f_{q_{0}}(x^{o}(t^{1}),u^{o}(t^{1}))]\rangle,\nnum\\\EN

and in the second case, (\ref{kir1}) and  (\ref{lllqq}) together yield
\EQ \label{llas2}\frac{d t_{s}(\epsilon)}{d \epsilon}|_{\epsilon=0}&\hspace{-.1cm}=&\hspace{-.1cm}\langle dN_{x^{o}(t^{-}_{s})},f_{q_{0}}(x^{o}(t^{-}_{s}),u^{o}(t^{-}_{s}))\rangle^{-1}\nnum\\&&\hspace{-.4cm}\times\langle dN_{x^{o}(t^{-}_{s})},T\Phi^{(t_{s},t^{1})}_{f_{q_{0}}}[f_{q_{0}}(x^{o}(t^{1}),u_{1})-f_{q_{0}}(x^{o}(t^{1}),u^{o}(t^{1}))]\rangle,\nnum\\\EN
where due to the transversality assumption in Definition \ref{d0},\\ $\langle dN_{x^{o}(t^{-}_{s})},f_{q_{0}}(x^{o}(t^{-}_{s}),u^{o}(t^{-}_{s}))\rangle\neq 0 $.

$\big($ We notice that (\ref{llas}) coincides with (\ref{kap}) since in the coordinate system given in the proof of Lemma \ref{44}, $dN_{x^{o}(t^{-}_{s})}=(0,...,0,1)$ therefore 
\EQ&&\langle dN_{x^{o}(t^{-}_{s})},f_{q_{0}}(x^{o}(t^{-}_{s}),u^{o}(t^{-}_{s}))\rangle^{-1}\langle dN_{x^{o}(t^{-}_{s})},T\Phi^{(t_{s},t^{1})}_{f_{q_{0}}}[f_{q_{0}}(x^{o}(t^{1}),u_{1})\nnum\\&&-f_{q_{0}}(x^{o}(t^{1}),u^{o}(t^{1}))]\rangle=f^{n^{-1}}_{q_{0}}(x(t_{s})).T^{n}\Phi^{(t_{s},t^{1})}_{ f_{q_{0}}}[f_{q_{0}}(x^{o}(t^{1}),u_{1})-f_{q_{0}}(x^{o}(t^{1}),u^{o}(t^{1}))].\big)\nnum \EN

Based on (\ref{lllq}) and (\ref{lllqq}), we have \\$\frac{d \Phi_{\pi, f_{q_{0}}}^{(t^{-}_{s}(\epsilon),t^{1}),x} }{d\epsilon}|_{\epsilon=0}\in  T_{x^{o}(t^{-}_{s})}\mathcal{S}$. The variation of the state trajectory at $t_{s}$ is obtained by evaluating $T\zeta$ on \\$(\frac{d t_{s}(\epsilon)}{d \epsilon}|_{\epsilon=0}).f_{q_{0}}(x^{o}(t^{-}_{s}),u^{o}(t^{-}_{s}))+T\Phi^{(t_{s},t^{1})}_{f_{q_{0}}}[f_{q_{0}}(x^{o}(t^{1}),u_{1})-f_{q_{0}}(x^{o}(t^{1}),u^{o}(t^{1}))]$, where by definition, $ T\zeta:T\mathcal{M}\rightarrow T\mathcal{M}$
is the push forward of $\zeta$.
Therefore
\EQ \hspace{-1cm}T\zeta\Bigg((\frac{d t_{s}(\epsilon)}{d \epsilon}|_{\epsilon=0}).f_{q_{0}}(x^{o}(t^{-}_{s}),u^{o}(t^{-}_{s}))&\hspace{-0cm}+&\hspace{-0cm}T\Phi^{(t^{-}_{s},t^{1})}_{f_{q_{0}}}[f_{q_{0}}(x^{o}(t^{1}),u_{1})\nnum\\&\hspace{-0cm}-&\hspace{-0cm}f_{q_{0}}(x^{o}(t^{1}),u^{o}(t^{1}))]\Bigg)\in T_{x^{o}(t_{s})}\mathcal{M}.\nnum\EN
Parallel to the results in \cite{Shaikh}, and following Lemma \ref{l555}, in case $(i)$, the state variation at  $t_{s}$ is
\EQ\label{kkk1}\frac{d\Phi_{\pi, f_{q_{1}}}^{(t_{s},t_{s}(\epsilon))} \circ \Phi_{\pi, f_{q_{0}}}^{(t^{-}_{s}(\epsilon),t^{1})}(x(t^{1}))}{d\epsilon}|_{\epsilon=0}&\hspace{-.1cm}=&\hspace{-.1cm}T\zeta\circ T\Phi^{(t^{-}_{s},t^{1})}_{f_{q_{0}}}[f_{q_{0}}(x^{o}(t^{1}),u_{1})-f_{q_{0}}(x^{o}(t^{1}),u^{o}(t^{1}))]\nnum\\&&\hspace{-3cm}+\frac{d t_{s}(\epsilon)}{d \epsilon}|_{\epsilon=0}[T\zeta(f_{q_{0}}(x^{o}(t^{-}_{s}),u^{o}(t^{-}_{s})))-f_{q_{1}}(x^{o}(t^{-}_{s}),u^{o}(t^{-}_{s}))]\in T_{x^{o}(t_{s})}\mathcal{M},\nnum\\\EN

and in case $(ii)$ 

\EQ \label{kkk2}\frac{d \Phi_{\pi, f_{q_{1}}}^{(t_{s}(\epsilon),t_{s})} \circ \Phi_{ \pi,f_{q_{0}}}^{(t^{-}_{s},t^{1})}(x(t^{1}))}{d\epsilon}|_{\epsilon=0}&\hspace{-.1cm}=&\hspace{-.1cm}T\zeta\circ T\Phi^{(t^{-}_{s},t^{1})}_{f_{q_{0}}}[f_{q_{0}}(x^{o}(t^{1}),u_{1})-f_{q_{0}}(x^{o}(t^{1}),u^{o}(t^{1}))]\nnum\\&&\hspace{-2.5cm}+\frac{d t_{s}(\epsilon)}{d \epsilon}|_{\epsilon=0}[f_{q_{1}}(x^{o}(t_{s}),u^{o}(t_{s}))-T\zeta(f_{q_{0}}(x^{o}(t^{-}_{s}),u^{o}(t^{-}_{s})))]\in T_{x^{o}(t_{s})}\mathcal{M}.\nnum\\\EN

Due to the sign change in (\ref{llas}) and (\ref{llas2}), both of the cases $(i)$ and $(ii)$ give the same results as in (\ref{kkk2}) and (\ref{kkk1}) respectively. Henceforth, we only consider the second case.
(\ref{llas2}) and (\ref{kkk2}) together imply
\EQ  \label{kkk3}\frac{d \Phi_{\pi, f_{q_{1}}}^{(t_{s}(\epsilon),t_{s})} \circ \Phi_{\pi, f_{q_{0}}}^{(t^{-}_{s},t^{1})}(x(t^{1}))}{d\epsilon}|_{\epsilon=0}&\hspace{-.1cm}=&\hspace{-.1cm}T\zeta\circ T\Phi^{(t^{-}_{s},t^{1})}_{f_{q_{0}}}[f_{q_{0}}(x^{o}(t^{1}),u_{1})-f_{q_{0}}(x^{o}(t^{1}),u^{o}(t^{1}))]\nnum\\&&\hspace{-4cm}+\langle dN_{x^{o}(t^{-}_{s})},f_{q_{0}}(x^{o}(t^{-}_{s}),u^{o}(t^{-}_{s}))\rangle^{-1}\langle dN_{x^{o}(t^{-}_{s})},T\Phi^{(t^{-}_{s},t^{1})}_{f_{q_{0}}}[f_{q_{0}}(x^{o}(t^{1}),u_{1})\nnum\\&&\hspace{-4cm}-f_{q_{0}}(x^{o}(t^{1}),u^{o}(t^{1}))]\rangle[f_{q_{1}}(x^{o}(t_{s}),u^{o}(t_{s}))-T\zeta(f_{q_{0}}(x^{o}(t^{-}_{s}),u^{o}(t^{-}_{s})))]\in T_{x^{o}(t_{s})}\mathcal{M},\nnum\\\EN
where $T\Phi^{(t_{f},t_{s})}_{f_{q_{1}}}\Big(\frac{d \Phi_{\pi, f_{q_{1}}}^{(t_{s}(\epsilon),t_{s})} \circ \Phi_{\pi, f_{q_{0}}}^{(t^{-}_{s},t^{1})}(x(t^{1}))}{d\epsilon}|_{\epsilon=0}\Big)\in K^2_{t_{f}}$ and by Lemma \ref{l6}, we have
\EQ\label{lll22}&& \hspace{-.7cm}0\leq \langle dh(x^{o}(t_{f})),T\Phi^{(t_{f},t_{s})}_{f_{q_{1}}}\Big(\frac{d \Phi_{\pi, f_{q_{1}}}^{(t_{s}(\epsilon),t_{s})} \circ \Phi_{\pi, f_{q_{0}}}^{(t^{-}_{s},t^{1})}(x(t^{1}))}{d\epsilon}|_{\epsilon=0}\Big)\rangle , \EN
therefore
\EQ\label{lll11}&& \hspace{-.7cm}0\leq \big\langle dh(x^{o}(t_{f})),T\Phi^{(t_{f},t_{s})}_{f_{q_{1}}}\Big\{T\zeta\circ T\Phi^{(t^{-}_{s},t^{1})}_{f_{q_{0}}}[f_{q_{0}}(x^{o}(t^{1}),u_{1})-f_{q_{0}}(x^{o}(t^{1}),u^{o}(t^{1}))]\nnum\\&&\hspace{-.2cm}+\langle dN_{x^{o}(t^{-}_{s})},f_{q_{0}}(x^{o}(t^{-}_{s}),u^{o}(t^{-}_{s}))\rangle^{-1}\nnum\\&&\hspace{-.2cm}\times\langle dN_{x^{o}(t^{-}_{s})},T\Phi^{(t^{-}_{s},t^{1})}_{f_{q_{0}}}[f_{q_{0}}(x^{o}(t^{1}),u_{1})-f_{q_{0}}(x^{o}(t^{1}),u^{o}(t^{1}))]\rangle\nnum\\&&\hspace{.1cm}[f_{q_{1}}(x^{o}(t_{s}),u^{o}(t_{s}))-T\zeta(f_{q_{0}}(x^{o}(t^{-}_{s}),u^{o}(t^{-}_{s})))]\Big\}\big\rangle, \nnum\\\EN
which implies

\EQ\label{lll1}&&\hspace{-.2cm}0\leq \langle T^{*}\Phi^{(t_{f},t_{s})}_{f_{q_{1}}}dh(x^{o}(t_{f})),T\zeta\circ T\Phi^{(t^{-}_{s},t^{1})}_{f_{q_{0}}}[f_{q_{0}}(x^{o}(t^{1}),u_{1})-f_{q_{0}}(x^{o}(t^{1}),u^{o}(t^{1}))]\nnum\\&&\hspace{-0cm}+\langle dN_{x^{o}(t^{-}_{s})},f_{q_{0}}(x^{o}(t^{-}_{s}),u^{o}(t^{-}_{s}))\rangle^{-1}\times\big\{\langle dN_{x^{o}(t^{-}_{s})},T\Phi^{(t^{-}_{s},t^{1})}_{f_{q_{0}}}\nnum\\&&\hspace{.3cm}[f_{q_{0}}(x^{o}(t^{1}),u_{1})-f_{q_{0}}(x^{o}(t^{1}),u^{o}(t^{1}))]\rangle\big\}\nnum\\&&\hspace{0cm}\times[f_{q_{1}}(x^{o}(t_{s}),u^{o}(t_{s}))-T\zeta(f_{q_{0}}(x^{o}(t^{-}_{s}),u^{o}(t^{-}_{s})))]\rangle.\nnum\\ \EN
By the linearity of push-forwards (see \cite{Lee2}), (\ref{lll1}) becomes
\EQ \label{lll111}&&\hspace{-.7cm} 0\leq\langle T^{*}\Phi^{(t_{f},t_{s})}_{f_{q_{1}}}dh(x^{o}(t_{f})),T\zeta\circ T\Phi^{(t^{-}_{s},t^{1})}_{f_{q_{0}}}[f_{q_{0}}(x^{o}(t^{1}),u_{1})- f_{q_{0}}(x^{o}(t^{1}),u^{o}(t^{1}))]\rangle\nnum\\&&\hspace{0cm}+ \langle dN_{x^{o}(t^{-}_{s})},f_{q_{0}}(x^{o}(t^{-}_{s}),u^{o}(t^{-}_{s}))\rangle^{-1}\nnum\\&&\hspace{0cm}\times\langle dN_{x^{o}(t^{-}_{s})},T\Phi^{(t^{-}_{s},t^{1})}_{f_{q_{0}}}[f_{q_{0}}(x^{o}(t^{1}),u_{1})-f_{q_{0}}(x^{o}(t^{1}),u^{o}(t^{1}))]\rangle\nnum\\&&\hspace{0cm}\times\langle T^{*}\Phi^{(t_{f},t_{s})}_{f_{q_{1}}}dh(x^{o}(t_{f})),f_{q_{1}}(x^{o}(t_{s}),u^{o}(t_{s}))- T\zeta(f_{q_{0}}(x^{o}(t^{-}_{s}),u^{o}(t^{-}_{s})))\rangle,\nnum\\\EN
where one may write this as  

\EQ\label{kk22}&&\hspace{-.7cm}0\leq\langle T^{*}\Phi^{(t^{-}_{s},t^{1})}_{f_{q_{0}}}\circ T^{*}\zeta\circ T^{*}\Phi^{(t_{f},t_{s})}_{f_{q_{1}}}dh(x^{o}(t_{f})),f_{q_{0}}(x^{o}(t^{1}),u_{1})- f_{q_{0}}(x^{o}(t^{1}),u^{o}(t^{1}))\rangle\nnum\\&&\hspace{0cm}+\mu \langle T^{*}\Phi^{(t^{-}_{s},t^{1})}_{f_{q_{0}}}dN_{x^{o}(t^{-}_{s})},f_{q_{0}}(x^{o}(t^{1}),u_{1})-f_{q_{0}}(x^{o}(t^{1}),u^{o}(t^{1}))\rangle,\nnum\\\EN
where
\EQ \label{mu}\mu&\hspace{-.1cm}=&\hspace{-.1cm}\langle dh(x^{o}(t_{f})),T\Phi^{(t_{f},t_{s})}_{f_{q_{1}}}[f_{q_{1}}(x^{o}(t_{s}),u^{o}(t_{s}))-T\zeta\big(f_{q_{0}}(x^{o}(t^{-}_{s}),u^{o}(t^{-}_{s}))\big)]\rangle \nnum\\&&\hspace{-.1cm}\times \langle dN_{x^{o}(t^{-}_{s})},f_{q_{0}}(x^{o}(t^{-}_{s}),u^{o}(t^{-}_{s}))\rangle^{-1}\in \mathds{R}.\nnum\\ \EN
Applying Proposition \ref{p1} to (\ref{kk22}) on $[t^{1}, t^{-}_{s}],$ we have
\EQ&&\hspace{-.7cm}\langle T^{*}\Phi^{(t^{-}_{s},t^{1})}_{f_{q_{0}}}\circ T^{*}\zeta\circ T^{*}\Phi^{(t_{f},t_{s})}_{f_{q_{1}}}dh(x^{o}(t_{f}))\nnum\\&&\hspace{-.7cm}+\mu T^{*}\Phi^{(t^{-}_{s},t^{1})}_{f_{q_{0}}}dN_{x^{o}(t^{-}_{s})},f_{q_{0}}(x^{o}(t^{1}),u^{o}(t^{1}))\rangle \nnum\\&&\hspace{-.7cm}\leq\langle T^{*}\Phi^{(t^{-}_{s},t^{1})}_{f_{q_{0}}}\circ T^{*}\zeta\circ T^{*}\Phi^{(t_{f},t_{s})}_{f_{q_{1}}}dh(x^{o}(t_{f}))\nnum\\&&\hspace{-.7cm}+\mu T^{*}\Phi^{(t^{-}_{s},t^{1})}_{f_{q_{0}}}dN_{x^{o}(t^{-}_{s})},f_{q_{0}}(x^{o}(t^{1}),u_{1})\rangle;\nnum\EN
then, as in the first step, define
\EQ \label{lamlam}p^{o}(t):=T^{*}\Phi^{(t^{-}_{s},t)}_{f_{q_{0}}}\circ T^{*}\zeta\circ T^{*}\Phi^{(t_{f},t_{s})}_{f_{q_{1}}}dh(x^{o}(t_{f}))\nnum\\+\mu T^{*}\Phi^{(t^{-}_{s},t)}_{f_{q_{0}}} dN_{x^{o}(t^{-}_{s})},\quad t\in [t_{0},t_{s}).\EN
Since $T^{*}\Phi^{(t^{-}_{s},t_{s})}_{f_{q_{0}}}=I$, choosing $t^{1}=t_{s}$ gives
\EQ p^{o}(t^{-}_{s})=T^{*}\zeta(p^{o}(t_{s}))+\mu dN_{x^{o}(t^{-}_{s})}.\EN
Following (\ref{ham2}) in the non-hybrid case, the Hamiltonian function is defined as 
\EQ H_{q_{0}}(x^{o}(t),p^{o}(t),u^{o}(t))&\hspace{-.1cm}=&\hspace{-.1cm}\langle \big\{T^{*}\Phi^{(t^{-}_{s},t)}_{f_{q_{0}}}\circ T^{*}\zeta\circ T^{*}\Phi^{(t_{f},t_{s})}_{f_{q_{1}}}dh(x^{o}(t_{f}))\nnum\\&+&\hspace{-.1cm}\mu T^{*}\Phi^{(t^{-}_{s},t)}_{f_{q_{0}}} dN_{x^{o}(t^{-}_{s})}\big\},f_{q_{0}}(x^{o}(t),u^{o}(t))\rangle,\quad  t\in[t_{0},t_{s}).\nnum\\\EN
\\
\textit{Step 3}:
We need to show $\lambda^{o}(t)=(x^{o}(t), p^{o}(t) )=(x^{o}(t),T^{*}\Phi^{(t_{f},t)}_{f_{q_{1}}}dh(x^{o}(t_{f}))),\hspace{.1cm} t\in[t_{0},t_{s})$  and \\$\lambda^{o}(t)=\Big(x^{o}(t),T^{*}\Phi^{(t^{-}_{s},t)}_{f_{q_{0}}}\circ T^{*}\zeta\circ T^{*}\Phi^{(t_{f},t_{s})}_{f_{q_{1}}}dh(x^{o}(t_{f}))+\mu T^{*}\Phi^{(t^{-}_{s},t)}_{f_{q_{0}}} dN_{x^{o}(t^{-}_{s})}\Big),\hspace{.1cm} t\in [t_{s},t_{f}]$ satisfy (\ref{hamham}).
By the definition of  Hamiltonian functions given by (\ref{ham1}) and (\ref{ham2}), it is obvious that $\dot{x}(t)=\frac{\partial H_{i}}{\partial p},\hspace{.2cm}i=0,1$.
To prove $\dot{p^{o}}(t)=-\frac{\partial H_{i}}{\partial x},\hspace{.2cm}i=0,1$, first we use the adjoint curve expression $\lambda(t), t\in [t_{s},t_{f}]$ given by (\ref{lam2}).  Therefore we have 
\EQ \dot{p^{o}}(t)=\frac{d}{dt}T^{*}\Phi^{(t_{f},t)}_{f_{q_{1}}}dh(x^{o}(t_{f})),\nnum\EN
where together with Lemma \ref{l3} and \ref{e}  implies
\EQ \label{ee} \dot{p^{o}}(t)=\left [-(\frac{\partial f^{i}_{q_{1}}}{\partial x^{j}}p^{j})\right ]^{n}_{i,j=1}=-\frac{\partial H_{q_{1}}(x^{o}(t),p^{o}(t))}{\partial x}.\EN
Same argument holds for $p^{o}(t)=T^{*}\Phi^{(t^{-}_{s},t)}_{f_{q_{0}}}\circ T^{*}\zeta\circ T^{*}\Phi^{(t_{f},t_{s})}_{f_{q_{1}}}dh(x^{o}(t_{f}))\nnum\\+\mu T^{*}\Phi^{(t^{-}_{s},t)}_{f_{q_{0}}} dN_{x^{o}(t^{-}_{s})}, \hspace{.4cm}t\in[t_{0},t_{s})$.
\\\\
\textit{Step 4}:
Here we complete the proof by obtaining the continuity of the Hamiltonian at the optimal switching time $t_{s}$. In \cite{Shaikh}, the Hamiltonian continuity based on the control needle variation approach is derived only for controlled switching hybrid systems. We give a continuity proof  in the case of autonomous switching hybrid systems via the following algebraic steps.\\
Notice that
\EQ \label{ee2}&&\hspace{-.5cm}\langle T^{*}\Phi^{(t_{f},t_{s})}_{f_{q_{1}}}dh(x^{o}(t_{f})), [f_{q_{1}}(x^{o}(t_{s}),u^{o}(t_{s}))-T\zeta(f_{q_{0}}(x^{o}(t^{-}_{s}),u^{o}(t^{-}_{s}))]\rangle =\nnum\\&&\hspace{-.5cm}\Big\langle T^{*}\Phi^{(t_{f},t_{s})}_{f_{q_{1}}}dh(x^{o}(t_{f})),\langle dN_{x^{o}(t^{-}_{s})},f_{q_{0}}(x^{o}(t_{s}),u^{o}(t_{s}))\rangle^{-1}\times\nnum\\&&\hspace{-.5cm}\langle dN_{x^{o}(t^{-}_{s})},f_{q_{0}}(x^{o}(t_{s}),u^{o}(t_{s}))\rangle\times\nnum\\&&\hspace{-.5cm}[f_{q_{1}}(x^{o}(t_{s}),u^{o}(t_{s}))-T\zeta f_{q_{0}}(x^{o}(t^{-}_{s}),u^{o}(t^{-}_{s}))]\Big\rangle.\nnum\\\EN
Therefore by (\ref{ee2}) we have
\EQ H_{q_{1}}(x^{o}(t_{s}),p^{o}(t_{s}),u^{o}(t_{s}))&\hspace{-.1cm}=&\hspace{-.1cm} \langle p(t_{s}),f_{q_{1}}(x^{o}(t_{s}),u^{o}(t_{s}))\rangle \nnum\\&\hspace{-.1cm}=&\hspace{-.1cm}\langle T^{*}\Phi^{(t_{f},t_{s})}_{f_{q_{1}}}dh(x^{o}(t_{f})),f_{q_{1}}(x^{o}(t_{s}),u^{o}(t_{s}))\rangle\hspace{.5cm}\mbox{by \ref{lam2}}\nnum\\&\hspace{-.1cm}=&\hspace{-.1cm}\langle T^{*}\Phi^{(t_{f},t_{s})}_{f_{q_{1}}}dh(x^{o}(t_{f})),T\zeta\big(f_{q_{0}}(x^{o}(t^{-}_{s}),u^{o}(t^{-}_{s}))\big)\rangle\nnum\\&+&\langle T^{*}\Phi^{(t_{f},t_{s})}_{f_{q_{1}}}dh(x^{o}(t_{f})),\nnum\\&&\langle dN_{x^{o}(t^{-}_{s})},f_{q_{0}}(x^{o}(t_{s}),u^{o}(t_{s}))\rangle^{-1}\nnum\\&\times&[f_{q_{1}}(x^{o}(t_{s}),u^{o}(t_{s}))-T\zeta\big(f_{q_{0}}(x^{o}(t^{-}_{s}),u^{o}(t^{-}_{s}))\big)]\rangle\nnum\\&\times&\langle dN_{x^{o}(t^{-}_{s})},f_{q_{0}}(x^{o}(t_{s}),u^{o}(t_{s}))\rangle\hspace{.5cm}\mbox{by \ref{ee2}}\nnum\\&\hspace{-.1cm}=&\hspace{-.1cm}\langle T^{*}\zeta \circ T^{*}\Phi^{(t_{f},t_{s})}_{f_{q_{1}}}dh(x^{o}(t_{f})),f_{q_{0}}(x^{o}(t^{-}_{s}),u^{o}(t^{-}_{s}))\rangle \nnum\\&\hspace{-.1cm}+&\hspace{-.1cm}\mu\langle dN_{x^{o}(t^{-}_{s})},f_{q_{0}}(x^{o}(t_{s}),u^{o}(t_{s}))\rangle \hspace{.5cm}\mbox{by \ref{mu}},\EN
hence by the definition of $p$ in (\ref{lam2}) and (\ref{lamlam}) we have
\EQ \langle p(t_{s}),f_{q_{1}}(x^{o}(t_{s}),u^{o}(t_{s}))\rangle =\langle p(t^{-}_{s}),f_{q_{0}}(x^{o}(t^{-}_{s}),u^{o}(t^{-}_{s}))\rangle,\nnum\EN
which gives the continuity of the Hamiltonian at the switching time $t_{s}$. \end{proof}
\\
It should be noted that setting $\zeta=I$ above subsumes the results obtained in \cite{Shaikh} for  non-impulsive autonomous hybrid systems.

 \bibliographystyle{siam.bst}
\bibliography{HSCC}
 
\end{document}